\documentclass[11pt]{article}
\usepackage{amsmath,amsthm, amssymb, amsfonts, enumitem, graphics}
\usepackage[sort&compress,numbers,comma,square]{natbib}
\usepackage[dvips]{color}
\usepackage[active]{srcltx}
\usepackage{hyperref}

\long\def\comment#1{}

\def\Grp#1{\left(#1\right)}

\def\Cbr#1{\left\{#1\right\}}

\def\Sbr#1{\left[#1\right]}

\def\cf#1{\mathbf{1}\Cbr{#1}}

\def\cum#1#2{{#1}_1+\cdots+{#1}_{#2}}           
\def\eno#1#2{{#1}_1, \ldots, {#1}_{#2}}         

\def\gv{\,|\,}

\def\toi{\to\infty}

\def\Lsup{\varlimsup}

\def\dd{\mathrm{d}}

\def\nth#1{\frac{1}{#1}}

\def\Reals{\mathbb{R}}

\numberwithin{equation}{section}
\newtheorem{theorem}{Theorem}[section]
\newtheorem{prop}[theorem]{Proposition}
\newtheorem{cor}[theorem]{Corollary}
\newtheorem{lemma}[theorem]{Lemma}
\newtheorem{definition}[theorem]{Definition}

\def\rx{\varepsilon}
\def\levy{\text{L\'evy}}
\def\pdf{g}    
\def\ldf{\varphi}  
\def\Ldf{\Pi}  
\def\fpt{\tau} 
\def\jp#1{\Delta_{#1}} 

\def\mean{\mathsf{E}}
\def\pr{\mathsf{P}}
\def\prs{(0,\infty)}
\def\nns{[0,\infty)}  
\def\dual#1{\overline{#1}}
\def\th{^{\rm th}}

\def\ddir{\text{Di}}
\def\dbeta{\text{Beta}}
\def\dexp{\text{Exp}}
\def\dgamma{\text{Gamma}}
\def\dunif{\text{Unif}}

\def\setto{\leftarrow}

\def\norm#1{\|{#1}\|}
\def\samtaq{samorodnitsky:taq:94}
\def\dfrac#1#2{{#1}/{#2}}
\def\dt{\downarrow}
\newenvironment{display}{
  \par\noindent\hrulefill
  \begin{enumerate}[topsep=0em, itemsep=0em, parsep=.25ex]
}{\end{enumerate}\vspace{-.5em}\hrulefill\medbreak}

\begin{document}
\begin{center}
  \large
  \textbf{On exact sampling of the first passage event of $\levy$
    process with infinite $\levy$ measure and bounded variation}  
  \\[1.5ex]
  \normalsize
  Zhiyi Chi \\
  Department of Statistics\\
  University of Connecticut \\
  Storrs, CT 06269, USA, \\[.5ex]
  E-mail: zhiyi.chi@uconn.edu \\[1ex]
  \today
\end{center}

\begin{abstract}
  We present an exact sampling method for the first passage event of a
  $\levy$ process.  The idea is to embed the process into another
  one whose first passage event can be sampled exactly, and then
  recover the part belonging to the former from the latter.  The
  method is based on several distributional properties that appear to
  be new.  We obtain general procedures to sample the first passage
  event of a subordinator across a regular non-increasing boundary,
  and that of a process with infinite $\levy$ measure, bounded
  variation, and suitable drift across a constant level or interval.
  We give examples of application to a rather wide variety of
  $\levy$ measures.


  \medbreak\noindent
\end{abstract}

\section{Introduction} \label{sec:intro}
The first passage event of a $\levy$ process is an important topic in
applied probability and has received extensive study \cite[cf.][and
references therein]{bertoin:96, bertoin:99, sato:99, doney:02,
  doney:06, huzak:04, kluppelberg:04, alili:05, eder:09}.  However,
although practically important and conceptually interesting, its 
exact sampling remains a subtle issue.  In particular, when the
$\levy$ measure of a process has infinite integral, except for a few
well-known cases, the distribution of the process is not analytically
available, which poses a significant hurdle to the exact sampling of
its first passage event.

Real-valued $\levy$ processes with bounded variation form a large
class.  Since each such process is the difference between two
independent subordinators, i.e., non-decreasing $\levy$ processes,
many properties of $\levy$ processes with bounded variation can be
deduced from those of subordinators.  By themselves, subordinators not
only play a significant role in the theory of $\levy$ processes
\cite{bertoin:96, doney:07, kyprianou:06}, but also have important
applications \cite{reynolds:71, kalbfleisch:78, hjort:90}.  As in the
general case, for most subordinators with infinite $\levy$ measure,
exact sampling methods for the first passage event are lacking,
although differential equations can be used to evaluate some 
parameters involved \cite{veillette:10:mcap, veillette:10:spl}.

In this paper, we present a method to sample the first passage event
for a rather wide range of real-valued $\levy$ processes with bounded
variation.  From now on, by sampling we always mean exact sampling.
The method can sample 1) the first passage event of a subordinator
across a regular non-increasing boundary, where the meaning of
``regular'' is specified in Section \ref{sec:fpe-sub}, 2) the first
passage event of a $\levy$ process with bounded variation and
non-positive drift across a positive constant level, and 3) the 
first exit event of a $\levy$ process with bounded variation and no
drift out of a closed interval with 0 as an internal point.  As a
by-product, the method also provides a way to sample infinitely
divisible random variables alternative to one in \cite{chi:12}.

We first give an overview of the method for subordinators, which is
the main issue, and then an overview of its extension to real-valued
$\levy$ processes with bounded variation.

\begin{figure}[t]
  \caption{\label{fig:sub-fpt}
    Sampling of the first passage event of a subordinator by embedding
    it into a ``carrier'' subordinator; see Section
    \ref{ssec:overview} for detail.
  }
  \begin{center}
    \setlength{\unitlength}{1mm}
    \begin{picture}(60,50)(2,-5)
      \put(0,0){\scalebox{.3}{\includegraphics{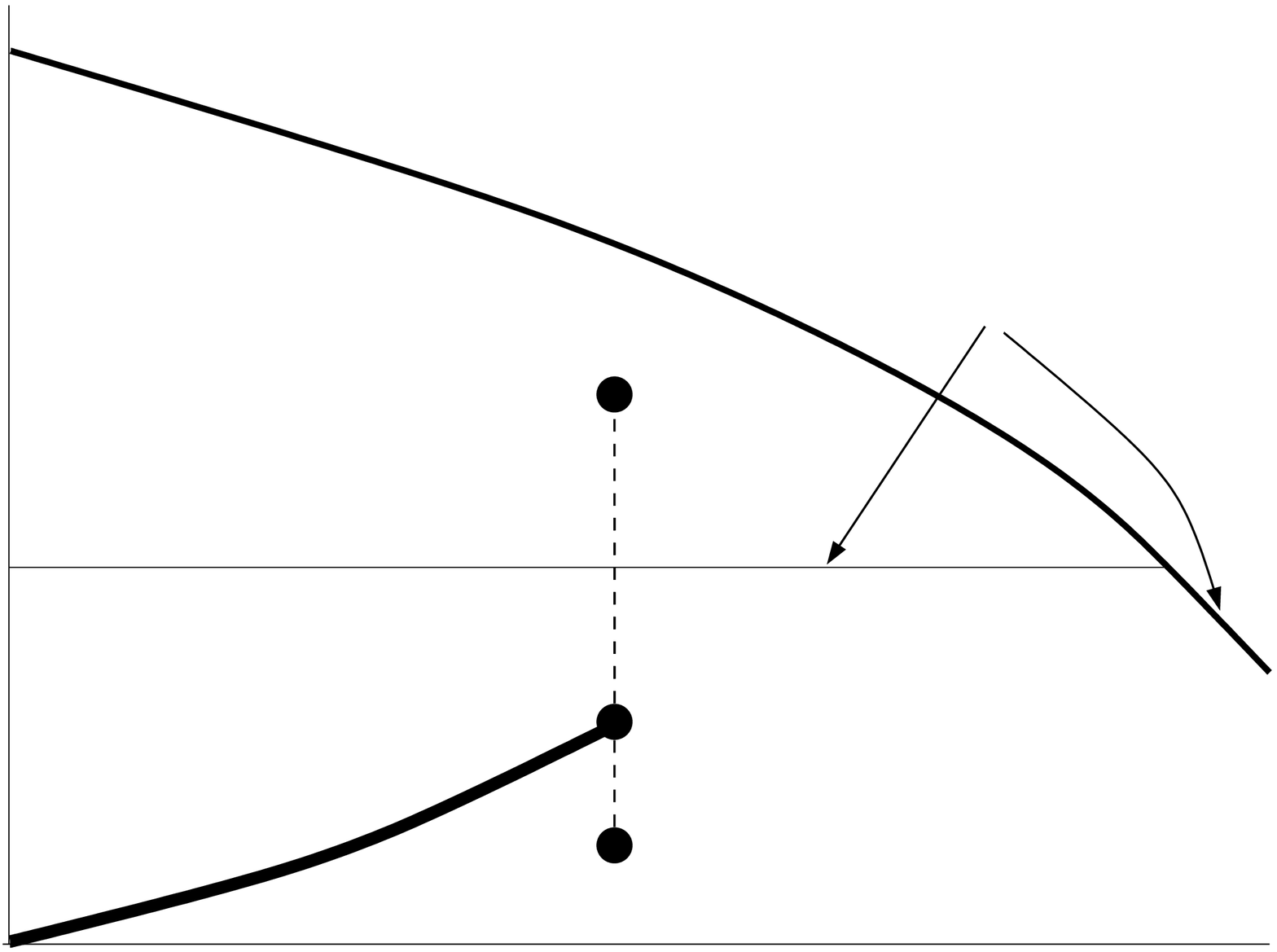}}}
      \put(31,19.5){$\jp S$}
      \put(31,10){$S(\fpt-)$}
      \put(20.5,27){$S(\fpt)$}
      \put(31,2){$Z(\fpt-)$}
      \put(10,35){$c(t)$}
      \put(-2,18){$r$}
      \put(60,2){$t$}
      \put(38,32){target boundary}
      \put(30,-5){(a)}
    \end{picture} \\[1em]
    \begin{picture}(60,43)(0,2)
      \put(0,0){\scalebox{.3}{\includegraphics{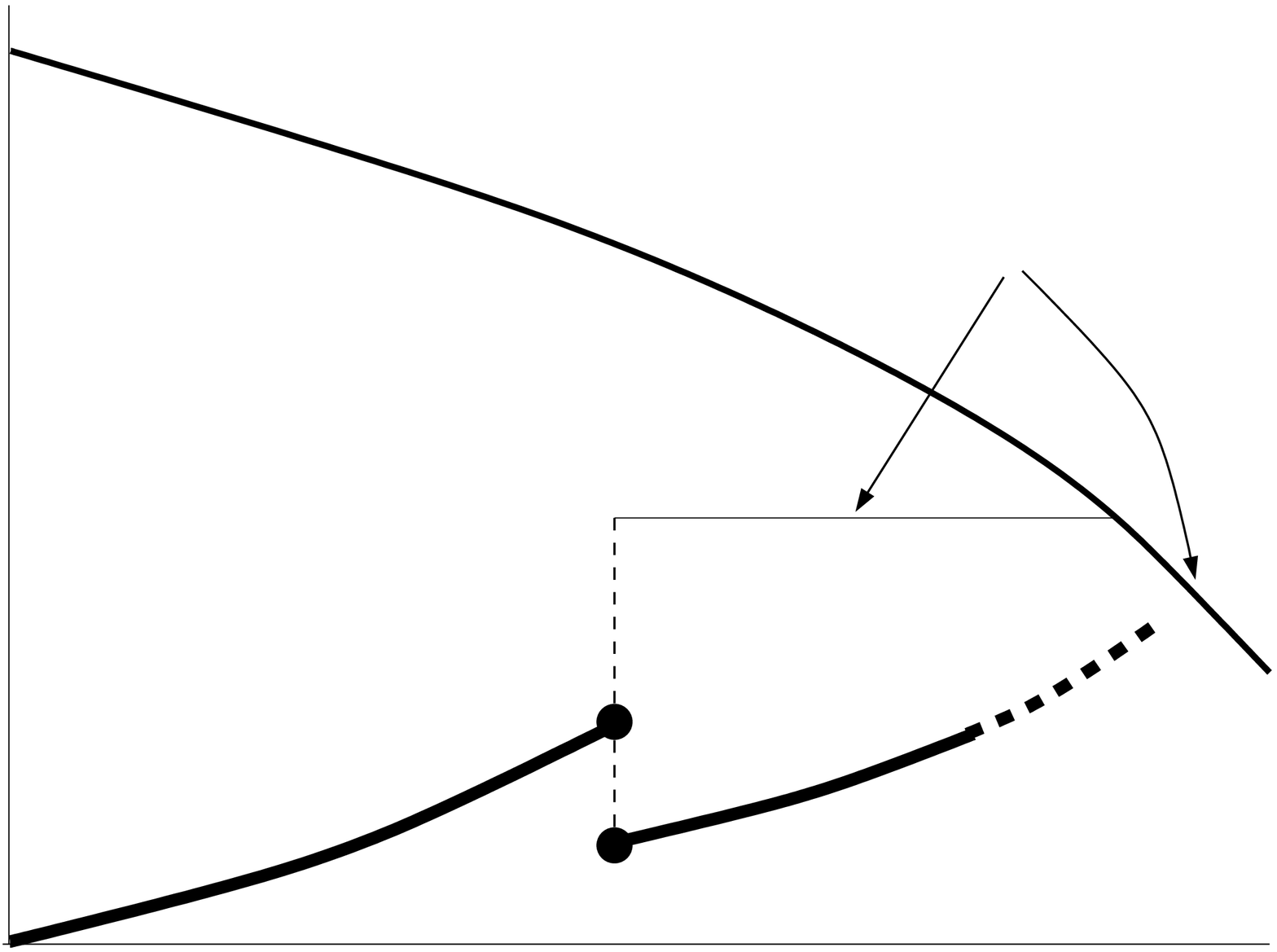}}}
      \put(5,20){$\jp S>r$ or}
      \put(5,14){$\jp S = \jp{X_2}$}
      \put(31,10){$S(\fpt-)$}
      \put(31,2){$Z(\fpt)=Z(\fpt-)$}
      \put(10,35){$c(t)$}
      \put(27,23){$r+Z(\fpt)$}
      \put(60,2){$t$}
      \put(34,34){target boundary}
      \put(30,-5){(b)}
    \end{picture}
    \hspace{4ex}
    \begin{picture}(60,43)(2,2)
      \put(0,0){\scalebox{.3}{\includegraphics{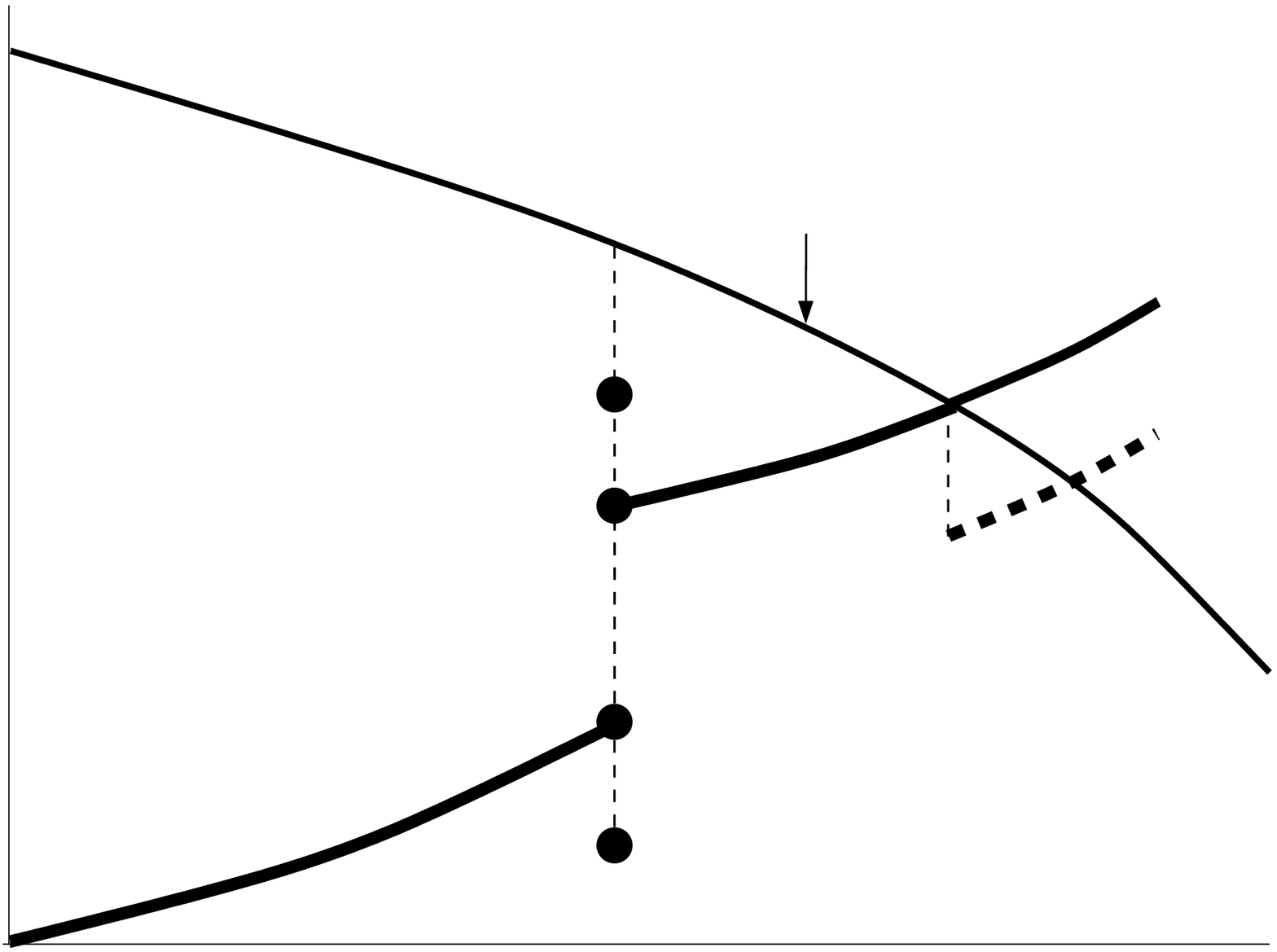}}}
      \put(5,14){$\jp S = \jp Z$}
      \put(20.5,27){$S(\fpt)$}
      \put(20.5,19){$Z(\fpt)$}
      \put(31,10){$S(\fpt-)$}
      \put(31,2){$Z(\fpt-)$}
      \put(10,35){$c(t)$}
      \put(60,2){$t$}
      \put(30,-5){(c)}
      \put(30,36){target boundary}
      \put(45,28){$\star$}
      \put(57,29){$A$}
      \put(57,22){$B$}
    \end{picture}
  \end{center}
\end{figure}

\subsection{Overview of method for subordinators}\label{ssec:overview}
In a nutshell, the idea is to embed a subordinator of interest into a
``carrier'' subordinator whose first passage event is tractable, and
to utilize the latter to sample for the former.  To put the idea into
perspective, think of a $\levy$ measure that can be
decomposed as
\begin{align}  \label{eq:h0}
  \ldf(x)\,\dd x + \chi(\dd x), \quad
  \text{with}\ \ \ldf(x) = \cf{0<x\le r}\gamma e^{-q x}
  x^{-\alpha-1},
\end{align}
where $r$, $\gamma>0$, $q\ge 0$, $\alpha\in (0,1)$, and $\chi$ is a
finite measure on $\prs$.  This type of $\levy$ measures coincide with
those that can be decomposed as $\tilde\ldf(x)\,\dd x + \chi(\dd x)$
with $\tilde\ldf(x)=(\gamma +O(x)) x^{-\alpha-1}$ as $x\dt 0$, and
give rise to many interesting processes discovered recently
\cite{caballero:10, kuznetsov:11, kyprianou:10, kuznetsov:12,
  kuznetsov:10, kuznetsov:10b}, such as Lamperti-stable process, whose
$\levy$ density is $\cf{x>0} e^{\beta x} (e^x-1)^{-\alpha-1}$, $\beta
< \alpha +1$.  For this particular process, we can set $r=\infty$ in
\eqref{eq:h0}.  However, in general, $r$ has to be finite.

A subordinator with $\levy$ density \eqref{eq:h0} can be represented
as the sum of two independent subordinators, one with $\levy$ density
$\ldf$, the other with $\levy$ measure $\chi$.  Since the latter is
compound Poisson and only poses minor problems, we shall ignore it 
altogether here.  Denoting by $Z$ the subordinator with $\levy$
density $\ldf$, it can be embedded into a stable subordinator as
follows.  Let $X_2$ and $X_3$ be independent subordinators
which are also independent from $Z$ and have $\levy$ densities
$\cf{0<x\le r} (1-e^{-q x}) x^{-\alpha-1}$ and $\cf{x>r}
x^{-\alpha-1}$, respectively.  Then $S = Z + X_2 + X_3$ is a stable
subordinator with index $\alpha$ \cite{\samtaq}.  One can expect that
the first passage event of $S$ is tractable, which indeed is the case.
Supposing the first passage event of $S$ is sampled, we then have to
recover the part due to $Z$ from the sample values.  This is the main
issue we have to address.

In general, let $Z$ be a subordinator and $c$ a non-increasing
function on $(0,\infty)$.   If $Z$ has positive drift $d>0$, then we
can consider $Z(t) - d t$ and $c(t) - d t$ instead.  Therefore,
without loss of generality, let $Z$ be a pure jump process.  Our
method requires some regularity of $c$.  For now we ignore the issue
and consider how to jointly sample the random variables $T = \Cbr{t>0:
  Z(t)>c(t)}$, $Z(T-)$, and $Z(T)$.  Suppose the $\levy$ measure of
$Z$ can be written as $\cf{0<x\le r} e^{-q x} \Lambda(\dd x)$, where
$0<r\le\infty$ and $q\ge 0$, such that the $\levy$ measure
$\Lambda(\dd x)$ gives rise to a subordinator $S$ whose first passage
event across any regular non-increasing boundary on $(0,\infty)$ can
be sampled.  Represent $S$ as $Z + X_2 + X_3$, such that $Z$, $X_2$
and $X_3$ are independent, with $X_2$ and $X_3$ having $\levy$
measures $\cf{0<x\le r} (1-e^{-q x}) \Lambda(\dd x)$ and $\cf{x>r}
\Lambda(\dd x)$, respectively.  Like $Z$, assume $X_2$ and $X_3$ are
pure jump processes.

The scheme of the method is shown in Fig.~\ref{fig:sub-fpt}.  To start
with, instead of $c(t)$, let $b(t)=c(t)\wedge r$ be the ``target
boundary'' for $S$ to cross and $\fpt$ the corresponding first
passage time.  By assumption, we can sample $(\fpt, S(\fpt-),
S(\fpt))$.  Observe the following simple but crucial fact: since
$S(\fpt-) \le b(\fpt)\le r$, $S$ can only have jumps of size no
greater than $r$ in $(0, \fpt)$.  Thus $Z(\fpt-) + X_2(\fpt-)
= S(\fpt-)$.  Now given $\fpt=t$ and $S(\fpt-)=s\le r$, we have to
sample $Z(\fpt-)$, which is possible for two reasons.  First,
the conditional distribution of $Z(\fpt-)$ is the same as that of
$Z(t)$ given $S(t)=s$, as if $t$ is fixed beforehand (cf.\ Section
\ref{sec:distribution}).  Second, using the properties of exponential
tilting and upper truncation of $\levy$ measures on $\prs$, the latter
conditional distribution can be sampled (cf.\
Section~\ref{sec:issues}).  In panel (a) of Fig.~\ref{fig:sub-fpt},
the sampled $Z(\fpt-)$ is less than $S(\fpt-)$.  However, since $X_2$
is compound Poisson, $Z(\fpt-)$ can be equal to $S(\fpt-)$ with
positive probability.  Having got $Z(\fpt-)$, we next sample
$Z(\fpt)$.  The jump of $S$ at $\fpt$ is $\jp S = S(\fpt) - S(\fpt-)$.
By independence, only one of $Z$, $X_2$ and $X_3$ can have a jump at
$\fpt$, so $\jp Z$ is either 0 or $\jp S$.  Fig.~\ref{fig:sub-fpt}
illustrates two scenarios.  If $\jp S>r$, then clearly it belongs to
$X_3$, giving $\jp Z=0$.  If $\jp S\le r$, then by comparing the
$\levy$ measures of $Z$ and $X_2$, with probability $e^{-q \jp S}$
(resp.\ $1-e^{-q \jp S}$), $\jp S$ belongs to $Z$ (resp.\ $X_2$),
giving $\jp Z=\jp S$ (resp.\ $\jp Z=0$) (cf.\ Sections
\ref{sec:distribution} and \ref{sec:fpe-sub}).  Thus $Z(\fpt) =
Z(\fpt-)+\jp Z$ can be sampled.  If $Z(\fpt)<c(\fpt)$, then by strong
Markov property, we can renew the procedure, but now with starting
time point at $\fpt$ and starting value of $S$ equal to $Z(\fpt)$.  As
can be expected, the procedure eventually stops, giving a sample value
of $(T, Z(T-), Z(T))$.

Note that, if $c$ is decreasing, then it is possible for $S$ to creep
across $c$, i.e., $\jp S=0$, as marked by $\star$ in panel (c).  In
this case, although $\jp Z=0$, if $q>0$, it is possible that $Z(\fpt)
< S(\fpt)$ and so the procedure has to be renewed; see the scenario
marked by $B$ in panel (c).  The characterization of creeping when $c$
is linear is known \cite{bertoin:96, griffin:11}.  However, the case
where $c$ is non-linear appears to be still unresolved.

To implement the scheme, one can first sample $\fpt$, then $(S(\fpt-),
S(\fpt))$ conditional on $\fpt$, and finally $(Z(\fpt-), Z(\fpt))$
conditional on $(\fpt, S(\fpt-), S(\fpt))$.  Among these samplings,
the one for $\fpt$ typically is the simplest.  The other samplings
require several theoretical results, which will be obtained in Section
\ref{sec:distribution}.  Finally, it is easy to introduce a terminal
point $K\le\infty$ and sample $(T', Z(T'-), Z(T'))$, with $T'=T\wedge
K$.  In particular, if $K=1$ and $c\equiv\infty$, the method samples
an infinitely divisible random variable with upper truncated $\levy$
measure $\cf{0<x\le r} e^{-q x} \Lambda(\dd x)$.  As a passing remark,
the so-called Vervaat perpetuity correspond to $q=0$ and
$\Lambda(\dd x) = a\,\dd x$, whose exact sampling is known
\cite{devroye:10, fill:10}.

\begin{figure}[t]
  \caption{\label{fig:nd-fpt}Sampling of the first passage event of 
    $Z=Z^+-Z^-$ across $a>0$, where $Z^\pm$ are independent
    subordinators.}
  \begin{center}
    \setlength{\unitlength}{1mm}
    \begin{picture}(96,55)(0,6.5)
      \put(-3,0){$O$}
      \put(0,0){\scalebox{.5}[.45]{\includegraphics{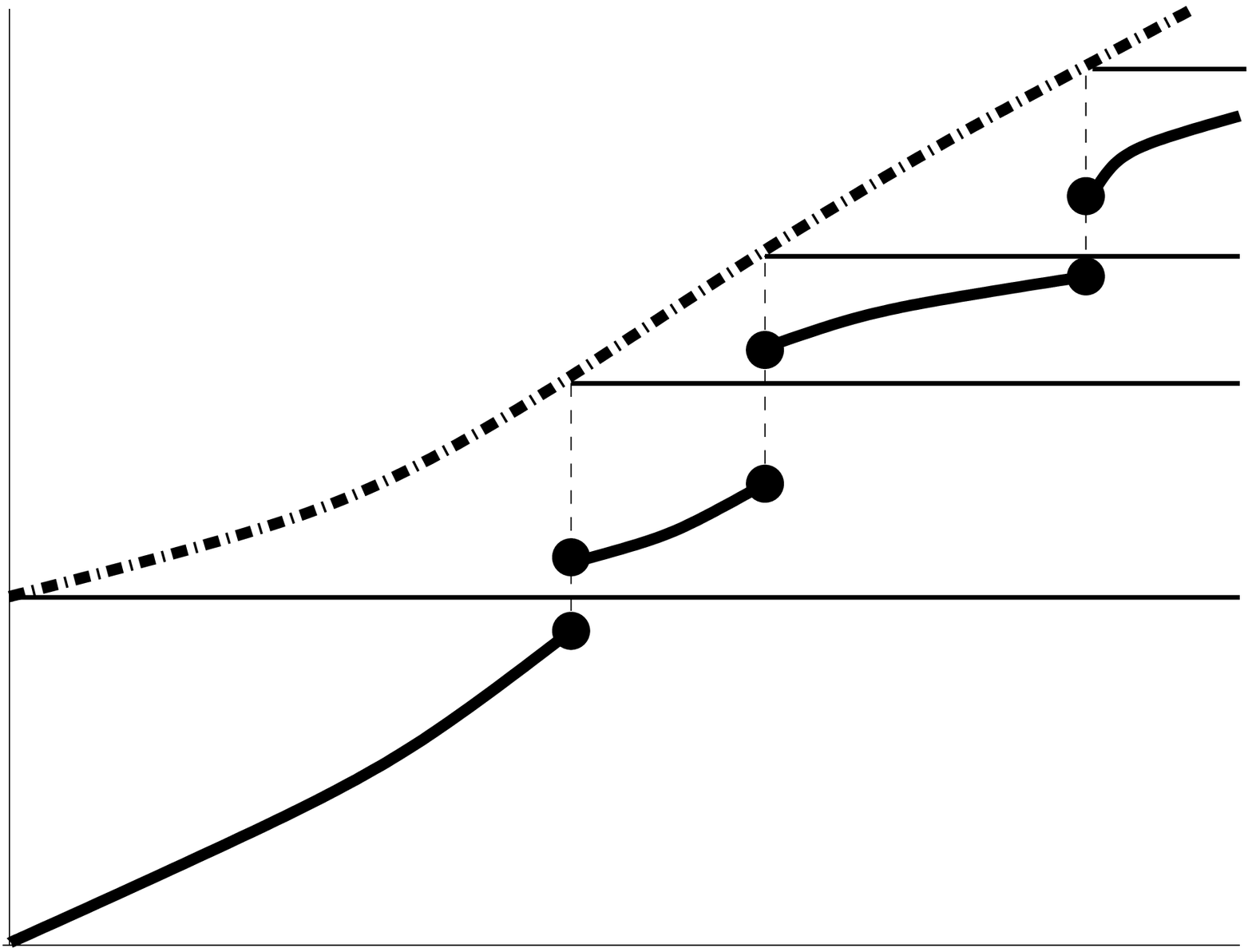}}}
      \put(-3,24){$a$}
      \put(40,16){$Z^+(\fpt_1-)$}
      \put(29,27){$Z^+(\fpt_1)$}
      \put(23,40){$a+Z^-(\fpt_1)$}
      \put(56,26){$Z^+(\fpt_2-)$}
      \put(39,50){$a+Z^-(\fpt_2)$}
      \put(78,40){$Z^+(\fpt_3-)$}
      \put(62,62){$a+Z^-(\fpt_3)$}
      \put(62,19){$1^{\rm st}$ target boundary}
      \put(62,34){$2^{\rm nd}$ target boundary}
      \put(90,1){$t$}
    \end{picture}
  \end{center}
\end{figure}

\subsection{Overview of extensions} \label{ssec:overview2}
The method described in Section \ref{ssec:overview} can be extended to
$\levy$ process with bounded variation, as each such process is the
difference of two independent subordinators.  Fig.~\ref{fig:nd-fpt}
illustrates the sampling of the first passage event across a constant
level $a>0$ by a $\levy$ process with \emph{non-positive\/} drift.
Write the process as $Z^+ - Z^-$, where $Z^\pm$ are independent
subordinators, with $Z^+$ having no drift.  The sampling can be
thought of as having $Z^+$ to ``catch up'' with $a+Z^-$.  To start
with, let $a$ be the target boundary for $Z^+$ to cross and $\fpt_1$
the corresponding first passage time.  It is evident that before
$\fpt_1$, $Z^+$ stays below $a+Z^-$.  However, at $\fpt_1$, since
$Z^+$ has a jump, it is possible for $Z^+$ to pass $a+Z^-$.   We can
use the method in Section \ref{ssec:overview} to sample $Z^+(\fpt_1)$.
Meanwhile, since $Z^-$ is independent of $\fpt_1$, we can use the
modification mentioned at the end of Section \ref{ssec:overview} to
sample $Z^-(\fpt_1)$.  If $a+Z^-(\fpt_1)< Z^+(\fpt_1)$, then $\fpt_1$
is the first passage time of $Z^+-Z^-$ across $a$.  If $a +
Z^-(\fpt_1)>Z^+(\fpt_1)$, then set $a + Z^-(\fpt_1)$ as the new target
boundary for $Z^+$ to cross, this time with $\fpt_1$ as the starting
time point and $Z^+(\fpt_1)$ the starting value for $Z^+$.  As long as
$\Lsup_{t\toi} Z(t) =\infty$ w.p.~1, the procedure eventually stops
(cf.\ Section \ref{sec:fpe-bv}).  Here the assumption that $Z$ has
non-positive drift is critical, since otherwise $Z^+$ can creep across
$a+Z^-$ with positive probability.  When this happens, the first
passage times of $Z^+$ across the target boundaries shown in
Fig.~\ref{fig:nd-fpt} will converge but never be equal to the first
passage time of $Z$ across $a$, resulting in the iteration going on
forever.  It can be seen that if $T=\inf\{t>0: Z^+(t) - Z^-(t)>a\}$,
then the procedure samples $(T, Z^+(T-), Z^+(T), Z^-(T))$.  As in
Section \ref{ssec:overview}, a terminal point $K\le \infty$ can be
introduced so that one can sample $(T', Z^+(T'-), Z^+(T'), Z^-(T'))$,
where $T' = T\wedge K$.  If $K<\infty$, the procedure eventually stops
w.p.~1, whether or not $\Lsup_{t\toi} Z(t)=\infty$.

\begin{figure}[t]
  \caption{\label{fig:bv-fpt}A ``phase plot'' for the sampling of the
  first exit event of $Z=Z^+-Z^-$ out of a fixed interval $[-a^-,
  a^+]$, where $a^\pm>0$ and $Z^\pm$ are independent subordinators.}
  \begin{center}
    \setlength{\unitlength}{.9mm}
    \begin{picture}(70,65)(2,5)
      \put(0,0) {
        \scalebox{.495}{\includegraphics{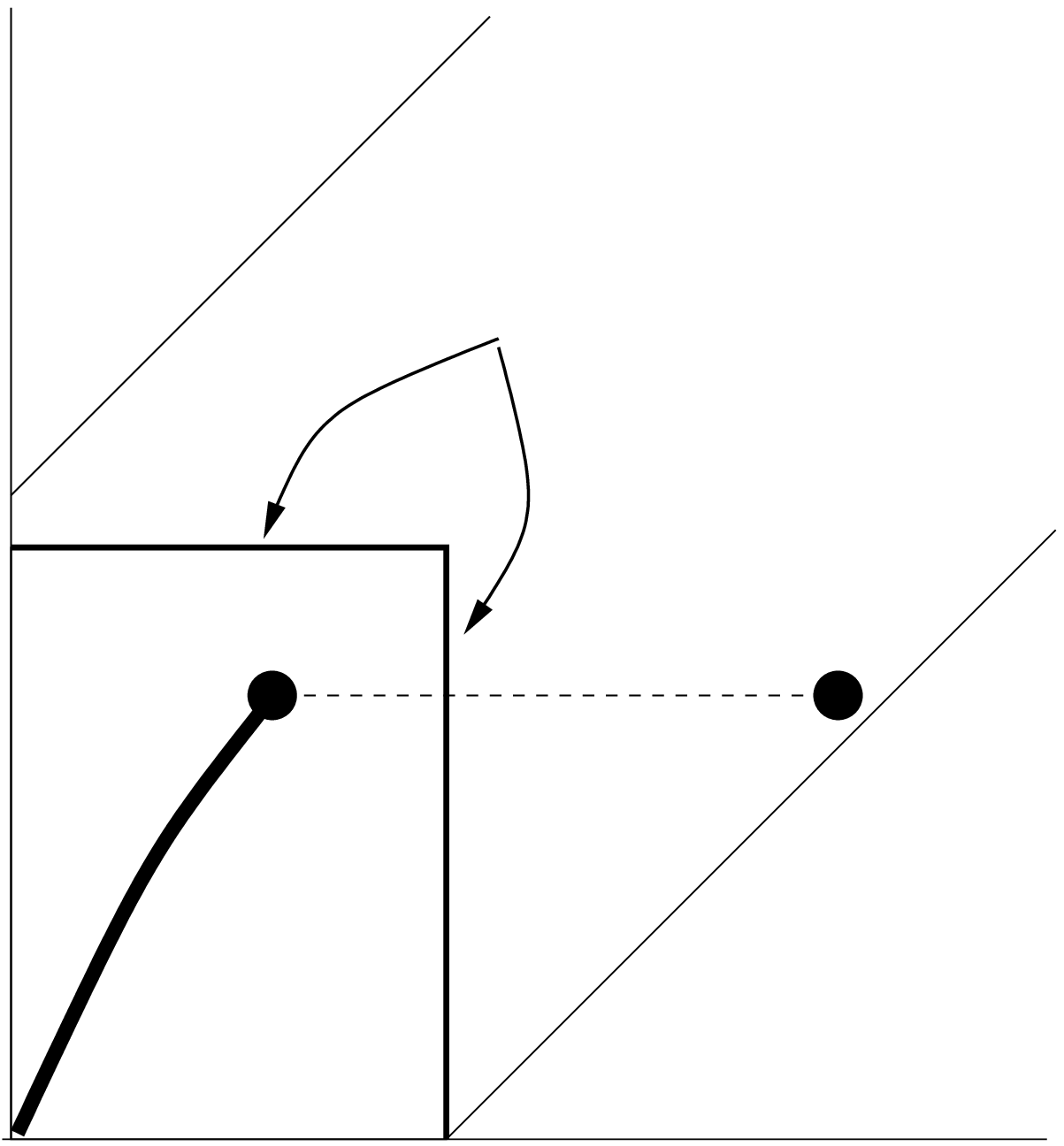}}
      }
      \put(5,49){\rotatebox{45}{$y-x=a^+$}}
      \put(42,8){\rotatebox{45}{$y-x=-a^-$}}
      \put(3,68){$y$}
      \put(63,1){$x$}

      \put(14,32){$S(\fpt-)$}
      \put(49,32){$S(\fpt)$}
      \put(31.5,24){$\jp{S^-}>r^-$}
      \put(31.5,19.5){$\jp{S^+}=0$}
      \put(22,53){target boundary}
      \put(3, 39){$a^+$}
      \put(24, 1.3){$a^-$}
      \put(35,-4){(a)}
    \end{picture}
    \begin{picture}(70,65)(-2,5)
      \put(0,0) {
        \scalebox{.495}{\includegraphics{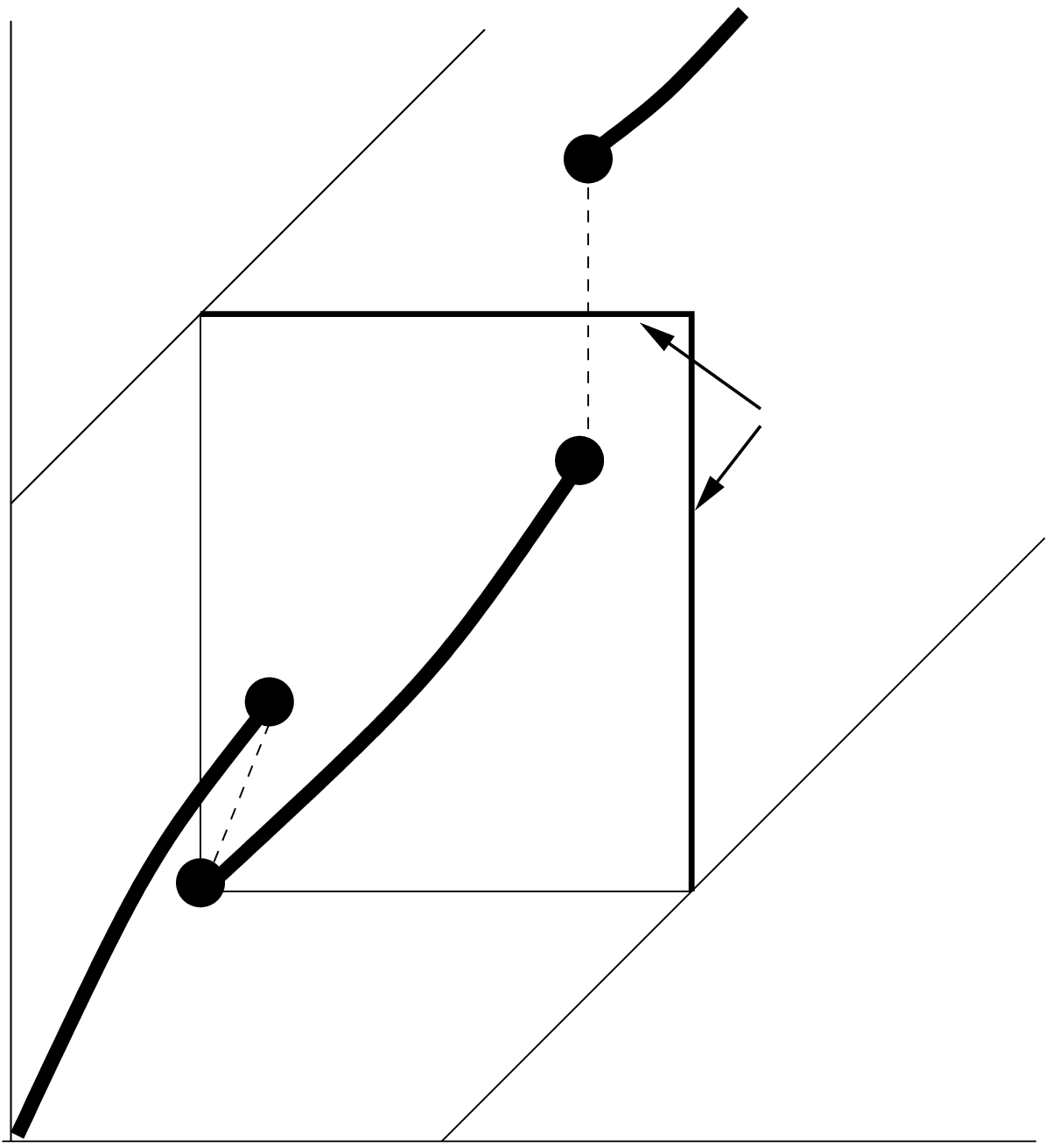}}
      }
      \put(5,49){\rotatebox{45}{$y-x=a^+$}}
      \put(42,8){\rotatebox{45}{$y-x=-a^-$}}
      \put(3,68){$y$}
      \put(63,1){$x$}
      \put(15,32){$S(\fpt-)$}
      \put(10,11){$W(\fpt) = W(\fpt-)$}
      \put(41,60){$W(\fpt')=S(\fpt')$}
      \put(51,48){target}
      \put(51,43){boundary}
      \put(3, 39){$a^+$}
      \put(26, 1.3){$a^-$}
      \put(35,-4){(b)}
    \end{picture}
  \end{center}
\end{figure}

It is also possible to sample the first exit event of $Z = Z^+-Z^-$
out of a fixed interval $[-a^-, a^+]$ with $a^\pm>0$ if it has no
drift.  To see how the method in Section \ref{ssec:overview} can
be extended to this case, it is convenient to consider the ``phase
plot'' of the $\levy$ process $W=(Z^-, Z^+)$, 
which shows the trajectory of $W$ on the $x$-$y$ plane without time 
axis.  Then the first exit of $Z$ out of $[-a^-, a^+]$ can be depicted
as the first exit of $W$ out of the band $\{(x,y): -a^- \le y-x \le
a^+\}$.  Suppose $Z^\pm$ each satisfies the assumption in
Section~\ref{ssec:overview}, so that, for example, $Z^+$ has $\levy$
measure $\cf{0<x \le r^+} \exp\{-q^+ x\} \Lambda^+(\dd x)$, where
$\Lambda^+(\dd x)$ is the $\levy$ measure of a subordinator $S^+$
whose first passage event across any constant level can be sampled.
To start with, let $b^-=a^-\wedge r^-$, $b^+=a^+\wedge r^+$, and set
the top and right sides of the rectangle $[0, b^-] \times [0, b^+]$ to
be the target boundary.  In panel (a) of Fig.~\ref{fig:bv-fpt},
$r^->a^-$ and $r^+<a^+$, resulting in the rectangle as shown.  Now
sample the first passage event of $S=(S^-, S^+)$ across the target
boundary.  To do this, we can first independently sample the first
passage times of $S^\pm$ across $b^\pm$.  If, as shown in panel (a),
$S^-$ makes a crossing at time $\tau$ before $S^+$, then sample
$(S^-(\fpt-), S^-(\fpt))$, and sample $S^+(\fpt-)=S^+(\fpt)$
conditional on $S^+(\fpt) < b^+$.  We next can use the method in
Section \ref{ssec:overview} to recover $Z^\pm(\fpt)$.  In the scenario
shown in Fig.~\ref{fig:bv-fpt}, since the jump of $S^-$ at $\fpt$
is greater than $r^-$, it is not part of $Z^-$, and so we end up with
$W(\fpt)$ as in panel (b).  The procedure is then renewed.  As
long as $Z\not\equiv 0$, the procedure will stop eventually and we get
the first exit event of $Z$ (cf.\ Sections~\ref{sec:distribution} and
\ref{sec:fpe-bv}).  It can be seen that if $T=\inf\{t>0: Z(t) \not\in
[-a^-, a^+]\}$, then the procedure samples $(T, Z^\pm(T-), Z^\pm(T))$.
As in Section \ref{ssec:overview}, a terminal point $K\le
\infty$ can be introduced so that one can sample $(T', Z^\pm(T'-),
Z^\pm(T'))$ with $T' =  T\wedge K$.

\subsection{Organization of the paper}
Section \ref{sec:prelim} fixes notation and recalls preliminary
results.  Section \ref{sec:distribution} obtains distributional
properties needed by the above methods.  The basic tools are results
on subordinators and fluctuation theory \cite[cf.][]{bertoin:96},
although some recent developments on the first passage event of a
general $\levy$ process could be exploited \cite[cf.][]{doney:07,
  doney:06, eder:09, kyprianou:06}.  In addition, we need to get some
detail on various conditional distributions of a subordinator as well
as on creeping.  Sections \ref{sec:fpe-sub}--\ref{sec:issues} present
procedures to implement the methods and show their validity, and
identify major sampling issues involved.  Sections
\ref{sec:stable}--\ref{sec:gamma} show examples of application of the
procedures to several types of $\levy$ measures.  The first type is
given in \eqref{eq:h0}.  The second type consists of finite sums of
$\levy$ measures of the form \eqref{eq:h0}, which require additional
techniques.  The third type consists of $\cf{0<x\le r} x^{-1}
e^{-x}\,\dd x + \chi(\dd x)$, with $\chi$ a finite measure on $\prs$.
This type gives rise to the aforementioned Vervaat perpetuity and the
Beta process in survival analysis, which has $\levy$ density $\cf{x>0}
e^{-c x}/(1-e^{-x})$ with $c>0$ and belongs to the Beta-class
processes \cite{hjort:90, kim:lee:03as, laud:98, kuznetsov:10}.
Rejection sampling and the Dirichlet distribution play an important
role in these examples \cite{devroye:86b, liu:01:book, glasserman:04}.

\section{Preliminaries} \label{sec:prelim}
For $x=(\eno x d)\in \Reals^d$, denote by $\norm x$ its $L^1$ norm
$\sum |x_i|$.  The 
convention $\inf\emptyset = \infty$ will be used 
all along; ``p.d.f.''\ will stand for ``probability density function
with respect to Lebesgue measure''.  For $a$, $b>0$, $\dbeta(a, b)$
denotes the distribution with p.d.f.\ $\cf{0<x<1}
x^{a-1}(1-x)^{b-1}/B(a,b)$, where $B(a,b) =
\Gamma(a)\Gamma(b)/\Gamma(a+b)$, $\dgamma(a,b)$ denotes the
one with p.d.f.\ $\cf{x>0} x^{a-1} e^{-x/b}/[b^a \Gamma(a)]$, and
$\dexp(b)$ denotes the one with p.d.f.\ $\cf{x>0}
e^{-x/b}/b$.  Finally, $\dunif(0,1)$ denotes the uniform distribution
on $(0,1)$.

Let $\tilde\nu$ and $\nu$ be two probability measures on a measurable
space $\Omega$ satisfying $\dd\tilde\nu/\dd\nu \propto \varrho\le C$,
where $\varrho\ge 0$ is a known function and $C>0$ a known constant.
Then $\tilde\nu$ can be sampled as follows: keep sampling $\xi\sim
\nu$ and $U\sim \dunif(0,1)$ until $C U\le \varrho(\xi)$.  This
procedure is the well-known rejection sampling, which is exact and
stops w.p.~1 \cite{devroye:86b, liu:01:book, glasserman:04}.

Let $\nu$ be an infinitely divisible distribution on $\prs$ with
$\levy$ measure $\Lambda$.  Given $q>0$, $\Lambda_q(\dd x) = e^{-q
  x}\Lambda(\dd x)$ is known as an exponentially tilted version of
$\Lambda$.  If $\nu_q$ is the infinitely divisible distribution with
$\levy$ measure $\Lambda_q$, then $\nu_q(\dd x) \propto e^{-q x}
\nu(\dd x)$ \cite{hougaard:86, aalen:92, brix:99, bertoin:96}.  By
setting $\varrho(x) = e^{-q x}\cf{x>0}$, rejection sampling can be
used to sample $\nu_q$ based on $\nu$.

The Dirichlet distribution $\ddir(\eno a k)$ with parameters $\eno a
k>0$ is a generalization of the Beta distribution.  It can be defined
as a distribution on $\Reals^k$, such that for any measurable function
$g\ge 0$ on $\Reals^k$ and $\omega\sim \ddir(\eno a k)$,
\begin{align*}
  \mean [g(\omega)]
  &
  =
  \frac{\Gamma(\cum a k)}{\Gamma(a_1)\cdots \Gamma(a_k)}
  \int \cf{\text{all}\ x_i\ge 0} g(x)
  \prod_{i=1}^k x_i^{a_i-1}\,\dd x_1\cdots \dd x_{k-1},
\end{align*}
where in the integral $x_k=1-x_1-\cdots-x_{k-1}$ instead of a variate,
and $x=(\eno x k)$.  The distribution has p.d.f.\ $\Gamma(\cum a
k)\prod_{i=1}^k x_i^{a_i-1}/\prod_{i=1}^k \Gamma(a_i)$ with respect to
the degenerate measure $\sigma_k(\dd x) = \cf{x_i\ge 0,\, \norm x =1}
\dd x_1 \cdots \dd x_{k-1}\, \delta(\dd x_k + \cum x {k-1}-1)$, where
$\delta$ is the Dirac measure at 0.  For convenience, we will refer to
it as the p.d.f.\ of $\ddir(\eno a k)$.  Also, if $k=1$, then for
$a>0$, define $\ddir(a)$ to be the Dirac measure at 1.

\section{Distributions of the first passage event}
\label{sec:distribution}
We need several distributional properties to implement the method
introduced in Section~\ref{sec:intro}.  The main issue is, provided a
subordinator $Z$ can be embedded into another subordinator $S =
Z+X_2+X_3$, how to recover the first passage event of $Z$ from the one
sampled for $S$.  As noted in Section \ref{ssec:overview}, we need to
take into account the possibility that $S$ creeps across a boundary.
Also, for both the method for subordinators and its extension to
$\levy$ processes, we need to make sure the corresponding sampling
procedures eventually stop.

\subsection{Results for subordinators}
Let $X=(\eno X d)$ be a $\levy$ process taking values in $\nns^d$ with
$\levy$ measure $\Ldf$ and Laplace exponent $\int (1-e^{-\langle
  \theta, x\rangle})\Ldf(\dd x)$, $\theta \in \nns^d$.  In the
application later, $X_1=Z$, $d=3$, and $X_i$ are independent.  By
definition, $\Ldf$ has no mass at $\{0\}$.  Let $\jp X$ be the jump
process of $X$.  The process $S = \norm X = X_1 + \cdots + X_d$ is a
subordinator with $\levy$ measure $\Ldf_S(\dd x) = \int_{\nns^d}
\cf{\norm z\in\dd x} \Ldf(\dd z)$, $x\in \prs$, and jump process $\jp
S = \norm{\jp X}$.  In the rest of this section, we shall always
assume
\begin{align} \label{eq:inf-levy}
  \Ldf_S\prs = \infty.
\end{align}
Denote $\dual\Ldf_S(x) = \Ldf_S(x,\infty)$, and given $c\in
C\prs$,
\begin{align} \label{eq:passage}
  \fpt_c = \fpt^S_c = \inf\Cbr{t>0: S(t)>c(t)}.
\end{align}
If $c(t)\equiv a>0$, the notation $\fpt_a$ will be used instead of
$\fpt_c$.  To implement the method introduced in Section
\ref{ssec:overview},  we will first sample $\fpt_c$, which is often
easy, then $(S(\fpt_c-), \jp S(\fpt_c))$ conditional on $\fpt_c$, and
finally $(X(\fpt_c-), \jp 
X(\fpt_c))$ conditional on $(\fpt_c$, $S(\fpt_c-)$, $\jp S(\fpt_c))$.
Among the results below, part 2) of Theorem \ref{thm:joint} will be
used for the conditional sampling of $(X(\fpt_c-), \jp X(\fpt_c))$,
while Theorem \ref{thm:conditional} for the conditional sampling of
$(S(\fpt_c-), \jp S(\fpt_c))$. 
\begin{theorem} \label{thm:joint}
  Let $c\in C\prs$ be non-increasing with $c(0+)>0$ and
  \eqref{eq:inf-levy} hold.

  1) Let $\Omega = \prs \times \nns^d \times
  (\nns^d\setminus\{0\})$.  For $(t,u,v)\in \Omega$,
  \begin{align}
    &
    \pr\Cbr{
      \fpt_c\in\dd t,\, X(\fpt_c-)\in\dd u,\, \jp X(\fpt_c)\in\dd v
    }
    \nonumber \\
    &\quad\qquad
    = \cf{0\le c(t) -\norm u<\norm v}
    \dd t\,\pr\Cbr{X(t)\in \dd u} \Ldf(\dd v),
    \label{eq:joint-jump} \\
    &
    \pr\Cbr{
      \fpt_c\in\dd t,\, X(\fpt_c-)\in\dd u,\, \jp X(\fpt_c)=0
    }
    \nonumber\\
    &
    \quad\qquad
    =\pr\Cbr{\fpt_c\in \dd t,\, S(\fpt_c)=c(\fpt_c)}
    \pr\Cbr{X(t)\in\dd u\gv S(t) = c(t)}.
    \label{eq:joint-creep}
  \end{align}
  
  2) For $t>0$, $s\in [0,c(t)]$, $z>c(t)-s$, and $u$, $v\in [0,\infty)^d$, 
  \begin{align}
    &
    \pr\Cbr{
      X(\fpt_c-)\in \dd u,\, \jp X(\fpt_c)\in \dd v \gv \fpt_c=t,
      S(\fpt_c-)=s, \jp S(\fpt_c)=z
    } \nonumber \\
    &
    \quad\qquad
    = \pr\Cbr{X(t)\in \dd u\gv S(t)=s} \Ldf_z(\dd v),
    \label{eq:joint2}
  \end{align}
  with $\Ldf_z(\dd v)=\pr\{V \in \dd v\gv \norm V = z\}$, where
  $V$ is a random vector following the distribution
  \begin{align*}
    \nu_a(\dd v) = 
    \dfrac{\cf{\norm v>a} \Ldf(\dd v)}{\dual\Ldf_S(a)},
  \end{align*}
  with $a\in (0,z)$ a fixed number.  The conditional probability
  measure $\Ldf_z(\dd v)$ is independent of the choice of $a\in
  (0,z)$.  Furthermore,
  \begin{align}
    \pr\Cbr{
      X(\fpt_c-)\in \dd u\gv \fpt_c=t,
      \jp S(\fpt_c)=0
    }
    = \pr\Cbr{X(t)\in \dd u\gv S(t)=c(t)}.
    \label{eq:joint2-creep}
  \end{align}
  
\end{theorem}

\begin{proof}
  1) It is clear that $0<\fpt_c<\infty$ w.p.~1.  We
  first show \eqref{eq:joint-jump}.  Following the proof for
  Proposition~III.2 in \cite{bertoin:96}, let $f\ge 0$ be a Borel
  function on $\Omega$ such that $f(t,u,v)=0$ when $\norm v =
  c(t)-\norm u$.   Then
  \begin{align} 
    f(\fpt_c, X(\fpt_c-), \jp X(\fpt_c))
    =
    \sum_t f(t, X(t-), \jp X(t))
    \cf{0\le c(t) - S(t-)<\norm{\jp X(t)}}.
    \label{eq:f-sum}
  \end{align}

  For each $t>0$, define function $H_t(v) = f(t, X(t-), v) \cf{0\le
    c(t)-S(t-)<\norm v}$ on $\nns^d$.  Since $H=(H_t)$ is a
  predictable process with respect to the filtration generated by $\jp
  X$, by \eqref{eq:f-sum} and the compensation formula,
  \begin{align*}
    &
    \mean[f(\fpt_c, X(\fpt_c-), \jp X(\fpt_c))]\\
    &
    \quad
    = \int_0^\infty \dd t\,
    \mean\Sbr{
      \int f(t, X(t-), v) \cf{0\le c(t)-S(t-)<\norm v} \Ldf(\dd v)
    } \\
    &
    \quad
    \stackrel{(a)}{=}
    \int_0^\infty \dd t
    \int f(t, u, v) \cf{0\le c(t) - \norm u<\norm v}
    \pr\Cbr{X(t)\in \dd u} \Ldf(\dd v) \\
    &\quad
    = 
    \int_\Omega\cf{0\le c(t) - \norm u<\norm v} f(t,u,v)
    \,\dd t\,\pr\Cbr{X(t)\in \dd u}  \Ldf(\dd v),
  \end{align*}
  where $(a)$ is due to $X(t-)\sim X(t)$ for any $t>0$.  Since $f$
  is arbitrary, this shows \eqref{eq:joint-jump} for $(t,u,v)\in
  \Omega$ with $\norm v\not= c(t) - \norm u$.  It remains to consider
  $(t,u,v)\in\Omega$ with $\norm v = c(t) 
  - \norm u$.  In this case, the right hand side of
  \eqref{eq:joint-jump} is 0.  If we define $f(t,u,v) = \cf{v =
    c(t)-u>0}$ for $t>0$, $u\ge 0$ and $v>0$, then by similar argument
  as above based on the compensation formula, but directly applied to
  $S$,
  \begin{align*}
    \pr\{S(\fpt_c-)<S(\fpt_c)= c(\fpt_c)\}
    = \int_0^\infty \dd t\, \int \pr\{S(t)\in \dd u\}
    \Ldf_S(\{c(t)-u\}).
  \end{align*}
  For each $t$, there is only a countable set of $u$ with
  $\Ldf_S(\{c(t) - u\})>0$.  On the other hand, under assumption
  \eqref{eq:inf-levy}, the distribution of $S(t)$ is continuous, i.e.,
  $\pr\{S(t) = u\}=0$ for any $u$ \cite[Theorem~27.4]{sato:99}.  As a
  result, $\int\pr\{S(t)\in \dd u\} \Ldf_S(\{c(t)-u\})=0$ for each
  $t$, and so the multiple integral is 0.  Finally, the proof of
  \eqref{eq:joint-jump} is complete by
  \begin{align} \label{eq:jump}
    \pr\Cbr{\jp X(\fpt_c)\not=0,\ S(\fpt_c)=c(\fpt_c)}
    = \pr\Cbr{S(\fpt_c-)<S(\fpt_c)=c(\fpt_c)} = 0.
  \end{align}

  Now consider \eqref{eq:joint-creep}.  Under \eqref{eq:inf-levy}, $S$
  is strictly increasing w.p.~1.  Clearly, $\jp X(\fpt_c)=0$ implies
  $S(\fpt_c)=c(\fpt_c)$.  On the other hand, from \eqref{eq:jump}, on
  the event $S(\fpt_c)=c(\fpt_c)$, $\jp X(\fpt_c)=0$ w.p.~1.  Define
  $\fpt_* = \inf\{t\ge 0: S(t)=c(t)\}$.  Then w.p.~1,
  \begin{align} \label{eq:cross}
    \{\fpt_* < \infty\} = \{\fpt_c = \fpt_*\} = \{S(\fpt_c) =
    c(\fpt_c)\}.
  \end{align}
  
  Let $f\ge 0$ be a Borel function on $\prs \times \nns^d$ with
  bounded support.  Then $\mean[f(\fpt_c,\, X(\fpt_c-)) \cf{ S(\fpt_c)
    = c(\fpt_c)}]$  can be expressed in two ways.  First,
  from \eqref{eq:jump}, it equals
  \begin{align}
    \int f(t,u) \cf{\norm u=c(t)}
    \pr\Cbr{\fpt_c \in\dd t,\, X(\fpt_c-)\in
      \dd u,\, \jp X(\fpt_c)=0}.  \label{eq:integral1}
  \end{align}
  Second, from \eqref{eq:jump} and \eqref{eq:cross}, the
  expectation also equals
  \begin{multline*}
    \mean[f(\fpt_c,\, X(\fpt_c))\cf{S(\fpt_c) = c(\fpt_c)}] 
    = \mean[f(\fpt_*,\, X(\fpt_*))\cf{\fpt_*<\infty}] \\
    =\int \mean[f(t,\, X(t))\gv \fpt_*=t]
    \,\pr\{\fpt_* \in \dd t\} 
    = \int f(t,u)\, \pr\{X(t)\in\dd u\gv \fpt_*=t\}
    \pr\{\fpt_*\in \dd t\}.
  \end{multline*}
  From the definition of $\fpt_*$ and \eqref{eq:cross}, the last
  integral is equal to 
  \begin{align} \label{eq:integral2}
    \int f(t,u)\, \pr\Cbr{\fpt_c\in \dd t,\, S(\fpt_c)=c(\fpt_c)}
    \pr\Cbr{X(t)\in \dd u\gv S(t)=c(t)}.
  \end{align}
  Since $f$ is arbitrary, comparing the integrals in
  \eqref{eq:integral1} and \eqref{eq:integral2} then yields
  \begin{align*}
    &
    \cf{\norm u=c(t)} \pr\Cbr{\fpt_c \in\dd t,\, X(\fpt_c-)\in
      \dd u,\, \jp X(\fpt_c)=0}\\
    &
    \quad\qquad
    = \pr\Cbr{\fpt_c\in \dd t,\, S(\fpt_c)=c(\fpt_c)}
    \pr\Cbr{X(t)\in \dd u\gv S(t)=c(t)}.
  \end{align*}
  Since the qualifier $\cf{\norm u=c(t)}$ can be
  removed from the identity, \eqref{eq:joint-creep} follows.

  2)  The process $(X,S)$ is a $\levy$ process with $\Ldf_{(X,S)}(\dd
  v, \dd z) = \Ldf(\dd v)\, \delta(\dd z -\norm v)$.  With similar
  argument as 1), for $t>0$, $s\ge 0$, $z>0$, $u\in \nns^d$, and
  $v\in \nns^d\setminus\{0\}$,
  \begin{align*}
    &
    \pr\Cbr{
      \fpt_c\in \dd t,\, X(\fpt_c-)\in \dd u,\ S(\fpt_c-)\in \dd s,\,
      \jp X(\fpt_c)\in \dd v,\, \jp S(\fpt_c)\in \dd z
    }
    \\
    &\quad\qquad=
    \cf{0\le c(t)- s<z}
    \dd t\,
    \pr\Cbr{X(t)\in \dd u, S(t)\in\dd s}  \Ldf(\dd v) \,
    \delta(\dd z - \norm v).
  \end{align*}
  On the other hand, applying \eqref{eq:joint-jump} directly to $S$,
  we get
  \begin{align*}
    &
    \pr\Cbr{
      \fpt_c\in \dd t,\, S(\fpt_c-)\in \dd s,\, \jp S(\fpt_c)\in
      \dd z
    }
    =
    \cf{0\le c(t) - s<z}
    \dd t\,\pr\Cbr{S(t)\in \dd s} \Ldf_S(\dd z).
  \end{align*}
  Therefore, in order to get \eqref{eq:joint2}, it suffices to show
  $\Ldf(\dd v) \delta(\dd z - \norm v) = \Ldf_z(\dd v) \Ldf_S(\dd z)$,
  which is equivalent to saying that for any measurable $A\subset
  (\nns^d\setminus\{0\}) \times \prs$,
  \begin{align*}
    \int \cf{(v,z)\in A} \Ldf(\dd v)\, \delta(\dd z - \norm v) = 
    \int \cf{(v,z)\in A} \Ldf_z(\dd v)\, \Ldf_S(\dd z),
  \end{align*}

  The left hand side is $\int \cf{(v, \norm v)\in A}\Ldf(\dd v)$.  To
  evaluate the right hand side, for any $\rx>0$, let $A_\rx =
  \{(v,z)\in A: z>\rx\}$.  Observe that $\cf{z\ge \rx}\Ldf_S(\dd
  z)/\dual\Ldf_S(\rx)$ is the distribution of $\norm V$ for
  $V\sim \nu_\rx$.  Then, from the definition of $\Ldf_z(\dd v)$,
  \begin{align*}
    &
    \int \cf{(v,z)\in A_\rx} \Ldf_z(\dd v) \Ldf_S(\dd z)\\
    &\qquad
    = \dual\Ldf_S(\rx) \int\cf{(v,z)\in A_\rx} \pr\{V\in \dd v\gv
    \norm V = z\} \pr\{\norm V\in \dd z\} \\
    &\qquad
    = \dual\Ldf_S(\rx) \int\cf{(v,\norm v)\in A_\rx} \pr\{V\in \dd
    v\}
    = 
    \int \cf{(v,\norm v)\in A_\rx} \cf{\norm v\ge \rx}\Ldf(\dd v).
  \end{align*}
  Let $\rx\dt 0$.  Since $A_\rx \uparrow A$, the last integral
  converges to $\int \cf{(v,\norm v)\in A}\Ldf(\dd v)$.  This
  completes the proof of \eqref{eq:joint2}.  Finally, since
  $\jp X(\fpt_c)=0$ if and only if $\jp S(\fpt_c)=0$, and by
  \eqref{eq:joint-jump}, $\pr\{\jp S(\fpt_c)\not=0,
  S(\fpt_c)=c(\fpt_c)\} = 0$, \eqref{eq:joint2-creep} follows from
  \eqref{eq:joint-creep}. 
\end{proof}

\begin{cor} \label{cor:time-density}
  For $a>0$ and $t>0$, define
  \begin{align} \label{eq:time-density}
    \psi_a(t) = \int_0^a \dual\Ldf_S(a-u) \pr\Cbr{S(t)\in\dd u}.
  \end{align}
  Then, under the same assumption as Theorem \ref{thm:joint},
  \begin{align}  \label{eq:time-density0}
    \pr\{\fpt_c\in\dd t,\, S(\fpt_c)>c(\fpt_c)\}
    = \psi_{c(t)}(t)\,\dd t = \dd t
    \int_0^{c(t)} \dual\Ldf_S(c(t)-u)\,\pr\Cbr{S(t)\in\dd u}.
  \end{align}
  In particular, for constant $a\in \prs$, $\fpt_a$ has p.d.f.\
  $\psi_a(t)$.
\end{cor}
\begin{proof}
  Apply \eqref{eq:joint-jump} in Theorem \ref{thm:joint} directly to
  $S$ to get 
  \begin{align*}
    \pr\{\fpt_c\in\dd t,\, S(\fpt_c)> c(\fpt_c)\}
    =
    \dd t\int\cf{0\le c(t)-u<v} \pr\{S(t)\in\dd u\}\Ldf_S(\dd v),
  \end{align*}
  which is \eqref{eq:time-density0}.  Since $\pr\Cbr{S(\fpt_a) >a} =
  1$ \cite[Theorem III.4]{bertoin:96}, $\pr\{\fpt_a\in\dd t\} =
  \psi_a(t)\,\dd t$.
\end{proof}
\begin{definition} \rm
  $S$ is said to satisfy the \emph{continuous density condition\/}, if
  $S(t)$ has a p.d.f.\ $\pdf_t$ on $\prs$ for each $t>0$ and the
  mapping $(t,x)\to \pdf_t(x)$ is continuous on $\prs\times \prs$.
\end{definition}

We next obtain the p.d.f.\ of $\fpt_c$ at the event that $S$ creeps
across a differentiable segment of $c$.  For linear $c$, the result is
shown in \cite{griffin:11}.  The following lemma is proved in
Appendix.
\begin{lemma} \label{lemma:cdc1}
  Under the continuous density condition on $S$, the mapping $(a,t)\to
  \psi_a(t)$ is continuous on $\prs\times \prs$, where $\psi_a(t)$ is
  the p.d.f.\ of $\fpt_a$ in \eqref{eq:time-density}.
\end{lemma}
\begin{prop} \label{prop:creep}
  Let $c\in C\prs$ be non-increasing with $c(0+)>0$.  If $c$ is
  differentiable on an open non-empty $G \subset\prs$ and $S$
  satisfies the continuous density condition, then
  \begin{align} \label{eq:creeping}
    \pr\Cbr{\fpt_c\in \dd t,\, S(\fpt_c) = c(\fpt_c)} = -c'(t)
    \pdf_t(c(t))\,\dd t, \quad t\in G.
  \end{align}
\end{prop}
\begin{proof}
  It suffices to consider $t\in G$ with $c(t) > 0$.  Given such $t$,
  $a:=c(t)$ is fixed.  Letting $q(\rx)= \pr\Cbr{t-\rx < \fpt_c\le
    t}$, for $\rx>0$, we have $q(\rx) = \pr\Cbr{S(t-\rx)<c(t-\rx),\,
    S(t)\ge a} = q_1(\rx)+q_2(\rx)$,
  where $q_1(\rx)  =\pr\Cbr{S(t-\rx)< a \le S(t)}$, $q_2(\rx) =
  \pr\Cbr{a \le S(t-\rx) < c(t-\rx)}$.  Let $\fpt_a = \inf\Cbr{s>0:
    S(s)>a}$.  Then $q_1(\rx) = \pr\Cbr{t-\rx < \fpt_a \le t}$.  By
  Corollary \ref{cor:time-density} and Lemma \ref{lemma:cdc1},
  $\psi_a$ is the continuous p.d.f.\ of $\fpt_a$, so the function
  $\pr\{\fpt_a\le \cdot\}$ is differentiable with derivative
  $\psi_a(t)$ at $t$.  Then $q_1(\rx)/\rx \to \psi_a(t) =
  \psi_{c(t)}(t)$ as $\rx\dt 0$.  On the other hand, $q_2(\rx) =
  \int_0^{c(t-\rx)-c(t)} \pdf_{t-\rx}(a+x)\,\dd x$.  Since $(t,x)\to
  \pdf_t(x)$ is continuous on $\prs \times \prs$ and $c$ is
  differentiable at $t$, $q_2(\rx)/\rx \to -c'(t) \pdf_t(c(t))$ as
  $\rx\dt 0$.  We thus get
  \begin{align*}
    \lim_{\rx\dt 0} \frac{\pr\Cbr{\fpt_c\le t} -
      \pr\Cbr{\fpt_c\le t-\rx}}{\rx} 
    = -c'(t) \pdf_t(c(t)) + \psi_{c(t)}(t).
  \end{align*}

  Similarly, as $\rx\dt 0$, $\rx^{-1} [\pr\Cbr{\fpt_c\le t+\rx} -
  \pr\Cbr{\fpt_c\le t}]$ has the same limit.  It follows that
  $\pr\Cbr{\fpt_c\le \cdot}$ is differentiable everywhere in the
  open set $\Cbr{t\in G: c(t)>0}$, and hence its derivative $-c'(t)
  \pdf_t(c(t)) + \psi_{c(t)}(t)$ is the p.d.f.\ of $\fpt_c$ on the set
  \cite[Theorem 7.21]{rudin:87}.  Then by Corollary
  \ref{cor:time-density},
  \begin{align*}
    \frac{\pr\Cbr{\fpt_c\in\dd t}}{\dd t} = -c'(t) \pdf_t(c(t))
    +
    \frac{\pr\Cbr{\fpt_c\in\dd t,\, S(\fpt_c)>c(\fpt_c)}}{\dd t},
  \end{align*}
  which yields \eqref{eq:creeping}.
\end{proof}

We need one more lemma before getting the second main result of this
section.
\begin{lemma} \label{lemma:cdc2}
  Let $S$ satisfy the continuous density condition.
  If $c$ is continuous and non-increasing on $\prs$ with
  $c(0+)>0$, then $\pr\Cbr{\fpt_c\in A}=0$ for any $A\subset \prs$
  with $\ell(A)=c(A)=0$, where $c(A)$ is the
  Riemann-Stieltjes integral $|\int \cf{x\in A\cap \prs} \,\dd
  c(x)|$. 
\end{lemma}
\begin{theorem} \label{thm:conditional}
  Let $c$ be an absolutely continuous non-increasing function on $\prs$
  with $c(0+)>0$.  Suppose $c$ is differentiable on $\prs\setminus F$
  for some closed set $F$ with $\ell(F)=0$.  Then under the continuous
  density condition on $S$, w.p.~1, $\fpt_c\in \prs\setminus F$ and
  for $u\in \nns^d$ and $v\in \nns^d\setminus\{0\}$,
  \begin{align} \label{eq:c-passage-jump}
    \pr\Cbr{X(\fpt_c-)\in\dd u,\, \jp X(\fpt_c)\in\dd v\gv \fpt_c}
    &= Z(\fpt_c)^{-1} \nu_1(\dd u,\dd v\gv \fpt_c), \\
    \pr\Cbr{X(\fpt_c-)\in\dd u,\, \jp X(\fpt_c)=0\gv \fpt_c}
    &= Z(\fpt_c)^{-1} \nu_2(\dd u\gv \fpt_c),
    \label{eq:c-passage-creep}
  \end{align}
  where for $t\in \prs\setminus F$,
  \begin{align*}
    \nu_1(\dd u, \dd v\gv t)
    &= \cf{0\le c(t)-\norm u<\norm v} \pr\{X(t) \in \,\dd u\}\,
    \Ldf(\dd v), \\
    \nu_2(\dd u\gv t)
    & 
    = - c'(t)\pdf_t(c(t))\,\pr\{X(t)\in \dd u\gv S(t) = c(t)\}, \\
    Z(t)
    &= -c'(t)\pdf_t(c(t)) + \int_0^{c(t)} 
    \dual\Ldf_S(c(t) - s)\pr\{S(t)\in \,\dd s\}.
  \end{align*}
\end{theorem}

\begin{proof}
  Since $c$ is absolutely continuous, $c(F)=0$, so by Lemma
  \ref{lemma:cdc2}, $\fpt_c\in \prs\setminus F$ w.p.~1.  By Theorem
  \ref{thm:joint} and Proposition \ref{prop:creep}, for $t\in
  \prs\setminus F$, $u\in \nns^d$, $v\in \nns^d\setminus\{0\}$,
  \begin{align*}
    &
    \pr\Cbr{
      \fpt_c\in\dd t,\, X(\fpt_c-)\in \dd u,\, \jp X(\fpt_c)\in
      \dd v
    } \\
    &\quad\qquad
    = 
    \cf{0\le c(t)-\norm u <\norm v} \dd t\,
    \pr\Cbr{X(t)\in\dd u} \Ldf(\dd v) = \dd t\,
    \nu_1(\dd u,\dd v\gv t),
    \\ 
    &
    \pr\Cbr{
      \fpt_c\in\dd t,\, X(\fpt_c-)\in \dd u,\, \jp X(\fpt_c)=0
    } \\
    &
    \quad\qquad
    = - c'(t) \pdf_t(c(t))\,\dd t\,
    \pr\{X(t)\in \dd u\gv S(t) = c(t)\}
    = \dd t\,\nu_2(\dd u\gv t).
  \end{align*}
  Integrate over $u$ and $v$ to get $\pr\Cbr{\fpt_c\in\dd t} =
  Z(t)\,\dd t$.  Then \eqref{eq:c-passage-jump} follows.
\end{proof}

\subsection{Results for general L\'evy processes with bounded variation}
For a process $X$ taking values in $\Reals$, denote
\begin{align} \label{eq:extreme}
  \overline X(t) = \sup\{X(s): s\le t\},
  \quad
  \underline X(t) = \inf\{X(s): s\le t\}.
\end{align}
The following results will be used to validate the extension described
in Section \ref{ssec:overview2}.
\begin{prop} \label{prop:bv}
  Let $X$ be a $\levy$ process taking values in $\Reals$ with bounded
  variation and non-positive drift.  Suppose $X$ is not compound
  Poisson.  Then for any $a>0$,
  \begin{align*}
    \pr\Cbr{\exists t>0 \text{ such that }
      \overline X(s)<a \text{ all } s<t, X(t-) = a \text{ or } X(t)=a}
    = 0.
  \end{align*}
\end{prop}
\begin{proof}
  For each $t>0$, denote $A_t=\{\overline X(s)<a \text{ all } s<t\}$.
  We first show
  \begin{align}
    &
    \pr\Cbr{\exists t>0 \text{ s.t.\ $X(t-)=a$, $\jp
        X(t)\not=0$}} \nonumber\\
    &\qquad\qquad
    = 0 = \pr\Cbr{\exists t>0 \text{ s.t.\ $A_t$, $X(t)=a$, $\jp
        X(t)\not=0$}}.  \label{eq:bv-1}
  \end{align}
  
  Let $\Ldf$ be the $\levy$ measure of $X$.  Given $\rx>0$,
  \begin{align*}
    \cf{\exists t>0 \text{ s.t.\ $X(t-)=a$, $|\jp X(t)|\ge
        \rx$}} \le
    \sum_{t: |\jp X(t)|\ge \rx} \cf{X(t-)=a}.
  \end{align*}
  Take expectation on both side and apply the compensation formula to
  get
  \begin{align*}
    &
    \pr\Cbr{\exists t>0 \text{ s.t.\ $X(t-)=a$, $|\jp X(t)|\ge
        \rx$}}\\
    &
    \quad
    \le \Ldf(\Reals\setminus (-\rx,\rx))
    \int_0^\infty \pr\Cbr{X(t-)=a}\,\dd t 
    \le \Ldf(\Reals\setminus (-\rx,\rx)) U(\{a\}),
  \end{align*}
  where $U(\cdot) = \int \pr\{X(t)\in\,\cdot\}\,\dd t$ is the
  potential measure of $X$.  Since $U$ is diffuse
  \cite[Proposition~I.15]{bertoin:96} and $\Ldf(\Reals
  \setminus(-\rx, \rx))<\infty$, the left hand side is 0 for any
  $\rx>0$, showing the first half of \eqref{eq:bv-1}.  On the other
  hand, since $A_t$ implies $X(t-)\le a$,
  \begin{align*}
    \pr\Cbr{\exists t>0 \text{ s.t.\ $A_t$, $X(t)=a$, $\jp X(t)\not=
        0$}} \le \pr\Cbr{\exists t>0 \text{ s.t. $\jp X(t) = a -
        X(t-)>0$}},
  \end{align*}
  which is 0 by the argument for Proposition~III.2 in
  \cite{bertoin:96}.  This shows the second half of \eqref{eq:bv-1}.

  To complete the proof, it only remains to show
  \begin{align*}
    \pr\Cbr{\exists t>0 \text{ s.t.\ $A_t$ and $X(t-)=X(t)=a$}} = 0,
  \end{align*}
  or equivalently, $\pr\Cbr{\fpt^*<\infty}=0$, where $\fpt^* =
  \inf\{t>0: A_t \text{ and } X(t-)=X(t)=a\}$.  Let $\fpt_a =
  \inf\{t>0: X(t)>a\}$.  Clearly, $\{\fpt^*<\fpt_a\} \subset
  \{\fpt^*<\infty\} \subset \{\fpt^* \le \fpt_a\}$.  Since $\{\fpt^* = 
  \fpt_a<\infty\}\subset \{X$ creeps across $a$ at $\fpt_a\}$, has 0
  probability \cite[cf.][Exercise VI.9]{bertoin:96}, $\pr\Cbr{\fpt^* <
    \infty} = \pr\Cbr{\fpt^*<\fpt_a}$.  Let $\eta\sim \dexp(1)$ be
  independent of $X$.  If $\pr\Cbr{\fpt^*<\fpt_a}>0$, then
  $\pr\Cbr{\fpt^* < \eta < \fpt_a}>0$ and hence $\pr\{\overline
  X(\eta)=a\}>0$.  However, from the fluctuation identity
  \cite[Theorem VI.5]{bertoin:96}, $\overline X(\eta)$ is either
  constant 0 or infinitely divisible with $\levy$ measure $\nu(\dd x)
  = \cf{x>0}\int_0^\infty t^{-1} e^{-t} \pr\Cbr{X(t)\in \dd x} \,\dd
  t$.   In the latter case, since $U$ is diffuse, $\nu$ is also
  diffuse, implying the distribution of $\overline X(\eta)$ is
  continuous on $\prs$ \cite[cf.][Remark 27.3]{sato:99}.  As a result,
  $\pr\{\overline X(\eta)=a\}=0$.  The contradiction implies
  $\pr\Cbr{\fpt^*<\infty}=0$.
\end{proof}

Applying the result to $X$ and $-X$ respectively and using union-sum
inequality, we get
\begin{cor} \label{cor:bv}
  Let $X$ be a $\levy$ process taking values $\Reals$ with bounded
  variation and no drift.  Suppose $X$ is not compound Poisson.  Then
  for any $a>0$, $b>0$,
  \begin{align*}
    \pr\Cbr{\exists t>0 \text{ such that }
      -b< \underline X(s)\le \overline X(s)<a \text{ all } s<t, X(t-)
      \text{ or } X(t) 
      = -b \text{ or } a} = 0. 
  \end{align*}
\end{cor}

\section{Sampling of first passage event for a subordinator}
\label{sec:fpe-sub}

We call a function $c$ ``regular'' if it satisfies the conditions in
Theorem \ref{thm:conditional}, i.e., $c$ is absolutely continuous and
non-increasing on $\prs$ with $c(0+) > 0$, and is differentiable on
$\prs\setminus F$, where $F$ is a closed set of Lebesgue measure 0.
Note that if $c$ is regular, then for any constant $a>0$, $c\wedge a$
is also regular.

Let $Z$ be a subordinator with $\levy$ measure $\Ldf$ and no drift,
such that 1) $\Ldf\prs=\infty$ and $\Ldf$ can be decomposed as
\begin{align} \label{eq:levy-decompose}
  \Ldf(\dd x) = e^{-q x}\cf{x\le r} \Lambda(\dd x) + \chi(\dd x),
  \quad q\ge 0, \quad r>0,
\end{align}
with $\chi\prs<\infty$, and 2) letting $S$ be a subordinator with
$\levy$ measure $\Lambda$ and no drift, its passage event across any
regular function can be sampled.

To utilize $S$ to sample the first passage event of $Z$ across a
regular boundary $c$, let $X_1$, $X_2$, $X_3$, and $Q$ be 
independent subordinators with no drift, and with $\levy$ measures
$e^{-q x} \cf{x\le r}\Lambda(\dd x)$, $(1-e^{-q x}) \cf{x\le r}
\Lambda(\dd x)$, $\cf{x>r}\Lambda(\dd x)$, and $\chi$, respectively.
Among the four, only $X_1$ is not compound Poisson.  Represent $Z$ and
$S$ as
\begin{align*}
  Z=X_1 + Q, \quad
  S=\norm X = X_1 + X_2 + X_3 \quad\text{with}\
  X =  (X_1, X_2, X_3)'.
\end{align*}
Denote $\fpt^Z_c = \inf\{t>0: Z(t)>c(t)\}$, $\fpt^S_c = \inf\{t>0:
S(t)>c(t)\}$, and $\jp Q$ the jump process of $Q$.  Table
\ref{a:sub-fpt} describes a general procedure to sample the first
passage event of $Z$ across $c$.  It actually does a little more.
Given a terminal point $0<K\le \infty$, it samples $(\fpt, Z(\fpt-),
\jp Z(\fpt))$, where $\fpt= \fpt^Z_c \wedge K$.  In particular, if
$c=\infty$ and $K=1$, the procedure samples an infinitely divisible
random variable with $\levy$ measure $\Ldf$.

\begin{table}[ht]
  \caption{\label{a:sub-fpt}
    Sampling of $(\fpt, Z(\fpt-), \jp Z(\fpt))$, where $\fpt =
    \fpt^Z_c\wedge K$, $c$ is a regular function or $\infty$, $0<K\le
    \infty$ (finite if $c\equiv\infty$)
  }
  \begin{display}
  \item[*\ ] Set $T=H=D=0$, $A=K$, $b(\cdot)\equiv c(\cdot)$
  \item If $D=0$, then sample $(D, J)\sim (\fpt_Q\wedge A, \jp
    Q(\fpt_Q\wedge A))$, where $\fpt_Q=\inf\{t: \jp Q(t)>0\}$.
  \item Sample $t_1\sim\fpt^S_{b\wedge r}$ and set $t = t_1\wedge D$.
  \item If $t=t_1<D$, then sample $(s,v)\sim (S(t-), \jp S(t))$
    conditional on $\fpt^S_{b\wedge r}=t$.
  \item If $t=D<t_1$, then sample $s\sim S(t)$ conditional on $S(t) <
    b(t)\wedge r$ and set $v=0$.
  \item Sample $x\sim X_1(t)$ conditional on $X_1(t) + X_2(t)=s$.
  \item If $v>0$, then sample $U\sim \dunif(0,1)$ and reset $v\setto v
    \cf{v\le r, U\le e^{-q v}}$.
  \item Update $T \setto T  +t$.  Set $\Delta=v+\cf{t=D} J$,
    $z=x+\Delta$, and update $H\setto H+z$. 
  \item If $z< b(t)$ and $t<A$, then update $A\setto
    A-t$, $D\setto D- t$, $b(\cdot)\setto b(\cdot+t)-z$, and go
    back to step 1 to start a new iteration; else output
    $(T,H-\Delta,\Delta)$ and stop.
  \end{display}
\end{table}

\begin{theorem} \label{thm:sampling1}
  If $c$ is a regular function and $0<K\le \infty$, then the procedure
  in Table \ref{a:sub-fpt} stops w.p.~1, and its output is a sample
  value of $(\fpt,\, Z(\fpt-),\, \jp Z(\fpt))$, where
  $\fpt=\fpt^Z_c\wedge K$.  The claim is still true if $c\equiv
  \infty$ and $K<\infty$.
\end{theorem}

\begin{proof}
  We only prove the claim for the case where $c$ is a regular
  function.  The case where $c\equiv \infty$ and $K<\infty$ can be
  similarly proved.

  Consider the first iteration.  With $A=K$, $D= t_Q \wedge K$.  Note
  $Z(t) = X_1(t)$ for $t<D$.  In step~2, with $b=c$, $t_1$ is a
  sample value of $\fpt^S_{c\wedge r}$ and $t$ that of $\fpt^* := 
  \fpt^S_{c\wedge r}\wedge \fpt_Q \wedge K$.  By independence,
  $t_1\not=D$ w.p.~1.  If $t_1>D$, then w.p.~1, $S(D-)=S(D)<c(D)\wedge
  r$.  Thus the pair $(s, v)$ generated by 
  steps~3 and 4 is a sample value of $(S(\fpt^*-), \jp S(\fpt^*))$
  conditional on $\fpt^*=t$.  Given $(\fpt^*, S(\fpt^*-), \jp
  S(\fpt^*))=(t, s, v)$, steps~5 and 6 sample $X_1(\fpt^*-)$ and
  $\jp 1(\fpt^*)$ from their joint conditional distribution, where
  $\jp 1$ is the jump process of $X_1$.  If $t=t_1<D$, i.e.\ $S$
  crosses $c\wedge r$ before $D$, then by part 2) of Theorem
  \ref{thm:joint}, $X_1(\fpt^*-)$ and $\jp 1(\fpt^*)$ are independent
  under the conditional distribution, following the distribution of
  $X_1(t)$ conditional on $S(t)=s$ and that of $\jp 1(t)$ conditional
  on $\jp S(t) = v$, respectively.  This is still true if $t=D<t_1$,
  as $X(D-)=X(D)$ and $\jp 1(D)=\jp S(D)=0$ w.p.~1.  By $s\le
  c(t)\wedge r\le r$, $\pr \{X_1(t)\in \cdot \gv S(t)=s\} =
  \pr\{X_1(t)\in \cdot \gv X_1(t)+X_2(t) =s\}$, hence the sampling
  of $x$ in step~5.  Clearly, $\jp S(t)=0$ implies $\jp 1(t)=0$.
  Suppose $\jp S(t)=v>0$.  The support of $\Ldf_X$ is within $\{(x_1,
  x_2, x_3): x_i\ge 0$, at most one is nonzero$\}$, such that for
  $y>0$, $\Ldf_X(\dd y \times \{0\}\times \{0\})= e^{-q y}\cf{y\le r}
  \Lambda(\dd y)$, $\Ldf_X(\{0\}\times\dd y\times\{0\}) = (1-e^{-q y})
  \cf{y\le r} \Lambda(\dd y)$, and $\Ldf_X(\{0\}\times \{0\}\times \dd
  y) = \cf{y>r}\Lambda(\dd y)$.  Then by Theorem \ref{thm:joint},
  $\pr\{\jp 1(t)\in\dd y\gv \jp S(t)=v\} = \Ldf_v(\dd y\times
  \{0\}\times \{0\}) = \cf{y=v\le r} e^{-q v}$, hence the updating
  of $v$ in step~6.  Put together, the triplet $(t, x, v)$ generated
  by the end of step~6 is a sample value of $(\fpt^*, X_1(\fpt^*-),
  \jp 1(\fpt^*))$, and $\Delta$ in step~7 is a sample
  value of $\jp Z(\fpt^*) = \jp 1(\fpt^*) + \jp Q(\fpt^*)$ and $z$
  that of $Z(\fpt^*)=X_1(\fpt^*-)+\jp Z(\fpt^*)$.  If the condition of
  termination is not satisfied, we can renew the sampling by shifting
  the origin to $(t, Z(t))$.  This justifies the updating of $A$ and
  $b$ in step~8.  Note that $D$ is the distance in time to the current
  jump of $Q$.  Once $D$ becomes 0, the next jump of $Q$ will be
  sampled.

  Let $T_0=0$, and for $n\ge 1$, $(T_n, H_n, \jp n)$ the value of
  $(T,H,\Delta)$ obtained by the end of the $n\th$ iteration.    By
  induction, we can make the following conclusion.  For $n\ge 1$,
  if $Z(T_{n-1})<c(T_{n-1})$ and $T_{n-1}<K$, then
  \begin{align} \label{eq:T-it}
    T_n = \inf\{t>T_{n-1}: S(t) - S(T_{n-1}) >[c(t) - Z(T_{n-1})]
    \wedge r \text{ or } \jp Q(t)>0\} \wedge
    K,
  \end{align}
  $H_n = Z(T_n)$ and $\jp n = \jp Z(T_n)$.  Evidently, $Z(T_n-) =
  H_n-\jp n$.

  To show that the procedure stops w.p.~1 and returns $(\fpt,
  Z(\fpt-), \jp Z(\fpt))$, it suffices to show $\pr\{T_n=\fpt$
  eventually$\}=1$.  It is clear that $T_0<\fpt$.  For $n\ge 1$, if
  $T_{n-1}<\fpt$, then, since $Z$ is strictly increasing w.p.~1,
  $Z(T_{n-1}) < Z(\fpt-)\le c(\fpt)\le c(T_{n-1})$.  Then by
  \eqref{eq:T-it}, $T_n >T_{n-1}$.  Observe that in this case, for any
  $t\in (T_{n-1}, T_n)$, 
  \begin{align*}
    Z(t) - Z(T_{n-1}) = X_1(t)- X_1(T_{n-1}) \le S(t) - S(T_{n-1}) \le
    c(t) - Z(T_{n-1}),
  \end{align*}
  giving $Z(t)\le c(t)$ and hence $T_n\le \fpt$.  Therefore, if $T_n
  \not= \fpt$ for all $n\ge 1$, then $T_n$ is strictly increasing and
  strictly less than $\fpt$.  Let $\theta=\lim T_n$.  Then $\theta\le
  \fpt<\infty$.  For $n\gg 1$, the compound Poisson processes $X_2$,
  $X_3$ and $Q$ make no jumps in $(T_n, \theta)$, and so $S(T_{n+1}) -
  S(T_n) = X_1(T_{n+1})-X_1(T_n) = Z(T_{n+1}) - Z(T_n)$.  Meanwhile,
  since $Z(T_n)<c(T_n)\le c(T_1)<\infty$, for $n\gg 1$, $0 <
  Z(T_{n+1}) - Z(T_n)<r$.  Then we get
  \begin{align*}
    r>Z(T_{n+1})-Z(T_n) = S(T_{n+1}) - S(T_n) \ge
    [c(T_{n+1})-Z(T_n)]\wedge r.
  \end{align*}
  It is easy to see that the inequalities imply $Z(T_{n+1})\ge
  c(T_{n+1})$.  The contradiction shows that w.p.~1, $T_n = \fpt$
  for some $n$.
\end{proof}

\section{Extensions to L\'evy processes with bounded variation}
\label{sec:fpe-bv}
\subsection{Non-positive drift, positive constant
  level} \label{ssec:fpe-nd}
Let $Z$ be a $\levy$ process with non-positive drift, such that its
$\levy$ measure $\Ldf$ satisfies
\begin{align} \label{eq:bv-condition}
  \Ldf(\Reals)=\infty, \quad 
  \int_{-\infty}^\infty (|x|\wedge 1) \Ldf(\dd x)<\infty.
\end{align}
Given constants $a>0$ and $0<K\le\infty$, let $\fpt^Z_a =
\inf\{t>0: Z(t)>a\}$ and $\fpt=\fpt^Z_a\wedge K$.  Decompose $Z$ as
$Z^+ - Z^-$, where $Z^\pm$ are independent subordinators with
$\levy$ measures
\begin{align} \label{eq:pn-levy}
  \Ldf^+(\dd x) = \cf{x>0}\Ldf(\dd x), \quad
  \Ldf^-(\dd x) = \cf{x>0} \Ldf(-\dd x),
\end{align}
respectively, with $Z^+$ having no drift.  Since the drift of $Z$ is
non-positive, if $\fpt_a^Z<\infty$, then w.p.~1, $Z$ makes a positive
jump at $\fpt_a^Z$ \cite[cf.][Exercise VI.9]{bertoin:96}.  Meanwhile,
$Z$ makes no jump at $K$ w.p.~1.  Therefore the only possible jump
that $Z$ can make at $\fpt$ is positive, giving $\jp Z(\fpt) =
\jp{Z^+}(\fpt)$, $Z^+(\fpt) = Z^+(\fpt-)+\jp Z(\fpt)$, and $Z^-(\fpt-)
= Z^-(\fpt)$.  We thus consider the sampling of $(\fpt, Z^+(\fpt-),
Z^-(\fpt), \jp Z(\fpt))$.  Table~\ref{a:nd-fpt} describes a procedure
to do this, which essentially follows the description in Section
\ref{ssec:overview2} but allows a terminal point $K\le \infty$ to be
included.  Note that by assumption, $Z^\pm$ cannot be both compound
Poisson.  If $\Ldf^+$ (resp.\ $\Ldf^-$) can be decomposed as in
\eqref{eq:levy-decompose}, then the procedure in Table~\ref{a:sub-fpt}
can be directly called in step 1 (resp.\ 2) in Table~\ref{a:nd-fpt}.
On the other hand, if one of $Z^\pm$ is compound Poisson, the
corresponding step is straightforward.

\begin{table}[ht]
  \caption{\label{a:nd-fpt}
    Sampling of $(\fpt, Z^+(\fpt-), Z^-(\fpt), \jp Z(\fpt))$, where
    $\fpt = \fpt^Z_a\wedge K$, $a$ is a positive constant, and $0<K\le
    \infty$
  }
  \begin{display}
  \item[*\ ] Set $T=H^+=H^-=0$, $A=K$, $b=a$
  \item Sample $(t, z^+, v)\sim (\fpt^+,  Z^+(\fpt^+-), \jp
    {Z^+}(\fpt^+))$, where $\fpt^+ = \fpt^{Z^+}_b \wedge A$.  Set
    $x=z^++v$.
  \item Sample $z^-\sim Z^-(t)$.
  \item Update $T\setto T + t$, $H^+\setto H^++x$, $H^- \setto
    H^- + z^-$. 
  \item If $x-z^-< b$ and $t<A$, then update $A\setto A-t$, $b\setto
    b+ z^- - x$, go back to step 1 to start a new independent
    iteration; else output $(T, H^+-v, H^-, v)$ and stop.
  \end{display}
\end{table}

\begin{prop} \label{prop:nd-fpt}
  Suppose $\Lsup_{t\toi} Z(t)=\infty$ w.p.~1 or $K<\infty$.  Then the
  procedure in Table~\ref{a:nd-fpt} stops w.p.~1, and its output is a
  sample value of $(\fpt, Z^+(\fpt-) , Z^-(\fpt), \jp Z(\fpt))$.
\end{prop}
\begin{proof}
  Let $T_0=0$, $H^+_0 = H^-=0$, and for $n\ge 1$, let $(T_n, H^+_n,
  H^-_n, v_n)$ be the value of $(T, H^+, H^-, v)$ at the end of
  the $n\th$ iteration.  By induction, for $n\ge 1$, the procedure has
  to continue into the $n\th$ iteration if and only if $Z(T_k)<a$ and
  $T_k<K$ for all $0\le k<n$, and in this case,
  \begin{align} \label{eq:bv-Tn-1}
    T_n = \inf\{t>T_{n-1}: Z^+(t) - Z^+(T_{n-1})>a - Z(T_{n-1})\}
    \wedge K > T_{n-1},
  \end{align}
  and $H^+_n = Z^+(T_n)$, $H^-_n = Z^-(T_n)$, $v_n = \jp Z(T_n)$.
  By $Z^+(T_n-) - Z^+(T_{n-1})\le a - Z(T_{n-1})$,
  \begin{align} \label{eq:z-bv-0}
    Z^+(T_n-) - Z^-(T_{n-1}) \le a.
  \end{align}
  We show that for $n\ge 1$, if the procedure has to continue into the
  $n\th$ iteration, then
  \begin{align} \label{eq:z-bv-1}
    \sup\{Z(t): t<T_n\} <a \qquad\text{w.p.~1}
  \end{align}
  Consider $n=1$.  Since $Z$ is right-continuous, there is $\rx>0$, 
  such that $\overline Z(\rx) <a$, where $\overline Z$ is defined as
  in \eqref{eq:extreme}.  Since at least one of $Z^\pm$ is strictly
  increasing, by \eqref{eq:z-bv-0}, for $\rx \le s\le t < T_1$, $Z(s)
  \le Z^+(t)-Z^-(\rx) < Z^+(T_n-)\le a$.  As a result, $\overline
  Z(t)<a$.  If \eqref{eq:z-bv-1} is not true, then there must be
  $Z(T_1-)=a$.  However, by Proposition \ref{prop:bv}, the probability
  for this to happen 0.  We then get \eqref{eq:z-bv-1} for $n=1$.  For
  $n\ge 2$, by renewal argument, if the procedure has to continue into
  the $n\th$ iteration, then $\sup\{Z(t) - Z(T_{n-1}): T_{n-1} \le t
  <T_n\} < a - Z(T_{n-1})$, which together with induction yields
  \eqref{eq:z-bv-1}.

  By assumption, $\fpt<\infty$ w.p.~1.  To finish the proof, it
  suffices to show w.p.~1, $T_n=\fpt$ eventually.  The compliment of
  the event has two cases.  The first one is that the procedure stops
  at the end of an iteration with $T_n \not= \fpt$.  In this case, by
  \eqref{eq:z-bv-1}, $T_n < \fpt \le K$.  Now $T_n<\fpt$ implies
  $Z(T_n)\le a$, while $T_n<K$ together with the stopping rule of the
  procedure implies $Z(T_n)\ge a$.  Thus $Z(T_n)=a$.  Also by
  \eqref{eq:z-bv-1}, $\overline Z(t)<a$ for all $t<T_n$.  However, by
  Proposition~\ref{prop:bv}, the probability of this case is 0.  The
  second case is that the procedure goes on forever.  In this case,
  by \eqref{eq:z-bv-1}, $Z(T_n)<a$ and $T_n<\fpt$ for all $n$.  Then
  by \eqref{eq:bv-Tn-1}, $T_n$ is strictly increasing with a limit
  $\theta\le \fpt$.  Now for any $t<\theta$, $\overline Z(t)<a$.
  Meanwhile, by $Z^+(T_{n+1}) - Z^+(T_n)\ge a - Z(T_n)>0$, letting
  $n\toi$ yields $Z(\theta-)=a$.  By Proposition \ref{prop:bv}, the
  probability for such $\theta$ to exist is also 0.
\end{proof}

\subsection{No drift, first exit out of an interval}
\label{ssec:fpe-bv-band}
Now consider the sampling of the first exit from an interval.  Let $Z$
be a $\levy$ process with \emph{no\/} drift, such that its $\levy$
measure $\Ldf$ satisfies \eqref{eq:bv-condition}.  Given constants
$a^\pm>0$ and $0<K\le \infty$, denote $I=[-a^-, a^+]$ and let
$\fpt^Z_I = \inf\{t>0: Z(t)\not\in I\}$, $\fpt = \fpt^Z_I \wedge K$.
As in Section \ref{ssec:fpe-nd}, write $Z = Z^+ - Z^-$, where $Z^\pm$ 
are independent subordinators with $\levy$ measures $\Ldf^\pm$,
respectively, and no drift.  Then $Z$ makes a positive jump if it
first exits $I$ at $a^+$, and a negative jump if it first exists $I$
at $-a^-$.  We thus consider the sampling of $(\fpt, Z^+(\fpt-),
Z^-(\fpt-), \jp{Z^+}(\fpt), \jp{Z^-}(\fpt))$.

Suppose that $\Ldf^\pm$ can be decomposed as in
\eqref{eq:levy-decompose}.  To be specific, for $\sigma\in \{\pm\}$,
\begin{align*}
  \Ldf^\sigma(\dd x) = \exp(-q^\sigma x) \cf{x\le r^\sigma}
  \Lambda^\sigma(\dd x) + \chi^\sigma(\dd x),
\end{align*}
where $q^\sigma\ge 0$ and $0<r^\sigma\le\infty$ are constants and
$\chi^\sigma\prs<\infty$, such that, letting $S^\sigma$ be a
subordinator with $\levy$ measure $\Lambda^\sigma$ and no drift, its
passage event across any positive constant level can be sampled.
In the case where $Z^\sigma$ is compound Poisson, we simply set
$S^\sigma\equiv 0$ and the time of the first passage of 
$S^\sigma$ across any positive boundary to be $\infty$.  Note that,
since $\Ldf(\Reals) = \infty$, at most one $Z^\sigma$ is compound
Poisson.

To utilize $S^\pm$ to sample  $(\fpt, Z^+(\fpt-), Z^-(\fpt-),
\jp{Z^+}(\fpt), \jp{Z^-}(\fpt))$, the idea is similar to that in
Section \ref{sec:fpe-sub}.  Table \ref{a:bv-fpt} describes a procedure
to do this.  In each iteration, we have to monitor two passage times,
i.e., the times when $S^\sigma$ cross $b^\sigma\wedge r^\sigma$,
$\sigma\in\{\pm\}$, respectively, where $b^\sigma$ is a constant
obtained from $a^\sigma$.  To simplify notation, denote by
$\fpt^S_{b\wedge r}(\sigma)$ the first passage time of $S^\sigma$
across $b^\sigma \wedge r^\sigma$.  Represent $S^\sigma = \norm
{X^\sigma} = X_1^\sigma + X_2^\sigma + X_3^\sigma$, where $X^\sigma =
(X_1^\sigma, X_2^\sigma, X_3^\sigma)'$ and $X_i^\sigma$ are
subordinators with $\levy$ measures $\exp(-q^\sigma x)\cf{x\le
  r^\sigma} \Lambda^\sigma(\dd x)$, $[1-\exp(-q^\sigma x)]\cf{x\le
  r^\sigma} \Lambda^\sigma(\dd x)$, and $\cf{x> r^\sigma}
\Lambda^\sigma(\dd x)$, respectively.  All of $X_i^\sigma$, $i\le 3$,
$\sigma\in\{\pm\}$ are assumed to be independent with no drift.

\begin{table}[ht]
  \caption{\label{a:bv-fpt}
    Sampling of $(\fpt, Z^+(\fpt-), Z^-(\fpt-), \jp {Z^+}(\fpt),
    \jp {Z^-}(\fpt))$, where $\fpt = \fpt^Z_I\wedge K$, $I = [-a^-,
    a^+]$, $a^-$, $a^+>0$ are constants, and $0<K\le \infty$
  }
  \begin{display}
  \item[*\ ] Set $T=H^+=H^-=D=0$, $A=K$, $b^+=a^+$,
    $b^-=a^-$.  Let $Q$ be a compound Poisson process with $\levy$
    measure $\chi^++\chi^-$.
  \item If $D=0$, then sample $(D, J)\sim (\fpt_Q\wedge A, \jp
    Q(\fpt_Q \wedge A))$ and set $J^+ = J\vee 0$, 
    $J^- = (-J)\vee 0$, where $\fpt_Q=\inf\{t: \jp
    Q(t)\not=0\}$.
  \item For $\sigma\in \{\pm\}$, sample $t^\sigma\sim
    \fpt^S_{b\wedge r}(\sigma)$.  Set $t = t^+\wedge
    t^-\wedge D$.  (Note: w.p.~1, $t^\pm$ and $D$ are different from
    each other.) 
  \item For $\sigma\in\{\pm\}$, sample $(x^\sigma, v^\sigma)\sim
    (X^\sigma_1(t-), \jp {X_1^\sigma}(t))$ 
    conditional on $\fpt^+\wedge \fpt^-\wedge D=t$, by applying steps
    3--6 in Table \ref{a:sub-fpt} to $X^\sigma$.
  \item Update $T\setto T +t$.  For $\sigma\in\{\pm\}$, set
    $\Delta^\sigma=v^\sigma + \cf{t= D} J^\sigma$, $z^\sigma
    =x^\sigma+\Delta^\sigma$, and update $H^\sigma\setto
    H^\sigma+z^\sigma$.
  \item If $z^+-z^-\in (-b^-, b^+)$ and $t<A$, then update $A\setto
    A-t$, $D\setto D-t$, $b^+ \setto b^+ + z^- - z^+$, $b^-
    \setto b^- + z^+ - z^-$, and go back to step 1 to start a new
    iteration; else output $(T,H^+-\Delta^+, H^--\Delta^-, \Delta^+,
    \Delta^-)$ and stop.
  \end{display}
\end{table}

\begin{prop} \label{prop:bv-fpt}
  Suppose $Z\not\equiv 0$.  Then the procedure in Table \ref{a:bv-fpt}
  stops w.p.~1, and its output is a sample value of $(\fpt,
  Z^+(\fpt-), Z^-(\fpt-), \jp{Z^+}(\fpt), \jp{Z^-}(\fpt))$.
\end{prop}

\begin{proof}
  First, $\fpt<\infty$ w.p.~1 as $\Lsup_{t\toi} |Z(t)|=\infty$
  \cite[Theorem VI.12]{bertoin:96}.  Let $T_0=0$, $H_0^+=H_0^-=0$,
  and for $n\ge 1$, let $(T_n, H_n^+, H_n^-, \Delta_n^+, \Delta_n^-)$
  be the value of $(T, H^+, H^-, \Delta^+, \Delta^-)$ obtained by the
  end of the $n\th$ iteration.  By induction and the same argument as
  in the proof of Theorem \ref{thm:sampling1}, for $n\ge 1$, the
  procedure has to continue into the $n\th$ iteration if and only if
  $Z(T_k)\in (-a^-,a^+)$ and $T_k<K$ for $0\le k<n$, and in this case,
  \begin{align*}
    T_n &=\inf\{t>T_{n-1}:
    S^+(t)-S^+(T_{n-1})> a^+-Z(T_{n-1}), \\
    &\qquad\quad\text{or } 
    S^-(t) - S^-(T_{n-1})> a^-+Z(T_{n-1}), \text{ or } \jp
    Q(t)>0\}\wedge K > T_{n-1}, \\
    H^+_n &= Z^+(T_n), \quad H^-_n = Z^-(T_n), \quad
    \Delta^+_n = \jp {Z^+}(T_n), \quad
    \Delta^-_n = \jp {Z^-}(T_n),
  \end{align*}
  and, as in the proof of Proposition \ref{prop:nd-fpt}, $Z^+(T_n-) -
  Z^-(T_{n-1})\le a^+ $, $\sup\{Z(t): t<T_n\} < a^+$, and likewise, by
  considering $-Z(t)$, $Z^-(T_n-)-Z^+(T_{n-1})\le a^-$, 
  $\inf\{Z(t): t<T_n\}>-a^-$.  The rest of the proof follows that of
  Proposition \ref{prop:nd-fpt}, except Corollary \ref{cor:bv} is
  used.
\end{proof}

\section{Sampling issues involved} \label{sec:issues}
The procedures in previous sections involve several types of
sampling, some of which are standard, while the others have to be 
dealt with on a case-by-case basis.  Consider the procedure in
Table~\ref{a:sub-fpt}.  The main task of its step~1 is
\begin{align} \label{eq:sampling-cp}
  \text{sample the first jump of a compound Poisson process on
    $\prs$}.
\end{align}
The rest of the procedure requires a subordinator $S$ with infinite
$\levy$ measure $\Lambda$ and analytically tractable properties be
available.  Under this prerequisite, given regular boundary $c$,
$t>0$, and $0<s\le r$, the main tasks of steps~2--5 are to sample
\begin{gather}
  \text{$\fpt^S_c=\inf\{t>0: S(t)>c(t)\}$},
  \label{eq:sampling-ft}
  \\
  \text{$(S(t-), \jp S(t))$, conditional on $\fpt^S_c=t$}, 
  \label{eq:sampling-c1}
  \\
  \text{$S(t)$, conditional on $S(t)<c(t)$, and}
  \label{eq:sampling-c2} \\
  \text{$X_1(t)$, conditional on $X_1(t) + X_2(t)=s$,}
  \label{eq:sampling-c3}
\end{gather}
respectively, where $X_1$, $X_2$ are independent subordinators with no
drift and with $\levy$ measures $\cf{x\le r} e^{-q x} \Lambda(\dd x)$
and $\cf{x\le r} (1-e^{-q x}) \Lambda(\dd x)$, respectively.  The
procedures in Tables \ref{a:nd-fpt} and \ref{a:bv-fpt} also boil down
to \eqref{eq:sampling-cp}--\eqref{eq:sampling-c3}.

For \eqref{eq:sampling-cp}, recall that if a compound Poisson process
$Q$ has $\levy$ measure $\chi\not=0$, then the time and size of its
first jump are independent with p.d.f.\ $q e^{-q t}\cf{t>0}$ and
distribution $\chi/q$, respectively, where $q=\int\dd\chi$ 
\cite[cf.][]{bertoin:96, sato:99}.  For complicated $\chi$, a
rejection sampling method known as ``thinning'' can be used 
\cite[cf.][]{devroye:86b, glasserman:04}.  Let $\mu$ be a $\levy$
measure such that 1) $\alpha=\int\dd\mu<\infty$ is readily available,
2) $\alpha^{-1}\mu$ is easy to sample, and 3) $\dd\chi =
\varrho\,\dd\mu$ for some function $\varrho\le 1$ that is easy to
compute.  Then the thinning based on $\mu$ is as follows.
\begin{display}
\item[* ] Set $t=0$
\item Sample $s$ from the distribution with p.d.f.\ $\alpha e^{-\alpha
    s}\cf{s>0}$.  Update $t\setto t+s$.
\item Sample $x$ from the distribution $\alpha^{-1}\mu$ and $U\sim
  \dunif(0,1)$.
\item If $U\le \varrho(x)$, then stop and output $(t,x)$ as a sample
  of the time and size of the first jump of $Q$; else go back to
  step~1.
\end{display}

For \eqref{eq:sampling-ft}, since $S$ is strictly increasing w.p.~1
and $c$ is non-increasing,
\begin{align} \label{eq:dist-ft}
  \pr\{\fpt_c^S \le t\} = \pr\{S(t)\ge c(t)\},
\end{align}
with each side being continuous and strictly increasing in $t>0$.
If the probability distribution of $S(t)$ is analytically available,
then $\fpt_c^S$ may be sampled by inversion method, i.e., sample
$U\sim\dunif(0,1)$ and return the unique value of $t$ satisfying
$\pr\{S(t)\ge c(t)\}=U$ \cite[cf.][]{devroye:86b}.  Alternatively, if
$S$ has scaling property, it can be utilized to sample $\fpt_c^S$.
Both possibilities are illustrated later.  The sampling for
\eqref{eq:sampling-c1} heavily relies on the results in Section
\ref{sec:distribution}, and its detail need be dealt with on a
case-by-case basis.  The sampling for \eqref{eq:sampling-c2} has the
following generic solution: keep sampling $x\sim S(t)$ until $x\le a$.
However, by utilizing the structure of $S(t)$, 
it is possible to make the sampling significantly more
efficient.

Finally, some comment on \eqref{eq:sampling-c3}.  Give $t>0$, for the
non-trivial case $q>0$, if $S(t)$ has a bounded p.d.f.\ $\pdf_t$,
then in principle rejection sampling can be used \cite{devroye:86b,
  hormann:04, liu:01:book, glasserman:04}.  Indeed, since the $\levy$
measure of $X_1(t)$ is $e^{-q x} \nu(\dd x)$, where $\nu(\dd x) = t
\cf{x\le r}\Lambda(\dd x)$ is that of $X_1(t)+X_2(t)$, $\pr\{X_1(t)\in
\dd x\} \propto e^{-q x} \pr\{X_1(t)  + X_2(t)\in \dd x\}$.  Recall
$S=X_1+X_2+X_3$, where $X_3$ has $\levy$ measure $\cf{x>r} \Lambda(\dd
x)$.  Since $X_3(t)$ is either 0 or $>r$, $X_1(t)+X_2(t)$ has
p.d.f.\ $\pdf_t(x)/\pr\{X_3(t)=0\}$ on $(0,r]$.  It follows that
$X_1(t)$ has a p.d.f.\ on $(0,r]$ in proportion to $e^{-q x}
\pdf_t(x)$, and  given $s\in (0,r]$,
\begin{align} \label{eq:tilting}
  \pr\{X_1(t)\in\dd x\gv X_1(t)+X_2(t)=s\}
  \propto e^{-q x} \pdf_t(x) \pr\Cbr{X_2(t)\in s-\dd x}.
\end{align}
Thus to sample \eqref{eq:sampling-c3}, we may keep sampling $v\sim
X_2(t)$ and $U\sim \dunif(0,1)$ until $v\le s$ and $\sup_x g_t(x)
\cdot U\le e^{-q(s-v)} \pdf_t(s-v)$ and then output $s-v$.  Here,
since $X_2(t)$ is compound Poisson, its sampling is standard
\cite{devroye:86b,   glasserman:04}.  However, a problem is that
$\pdf_t$ can be hard to evaluate.  To get around the problem, the
structure of $S(t)$ needs to be exploited.

\section{Exponentially tilted upper truncated stable
  L\'evy density} \label{sec:stable}
Consider the measure $\Ldf(\dd x) = \cf{x\le r} e^{-q x} \Lambda(\dd
x) +\chi(\dd x)$ specified in \eqref{eq:h0}, where $\Lambda(\dd x) =
\cf{x>0}\gamma x^{-1-\alpha}\,\dd x$ with $\alpha\in (0,1)$.  Let $c$
be a regular function as defined in Section \ref{sec:fpe-sub}.  As an 
application of the procedure in Table \ref{a:sub-fpt}, we next show an
algorithm to sample the first passage event of a subordinator $Z$ with
$\levy$ measure $\Ldf$ and no drift across $c$.  Using the
procedure in Table \ref{a:nd-fpt} or \ref{a:bv-fpt}, an algorithm can
be similarly devised to sample the first passage event of a $\levy$
process with non-positive or no drift across a constant level or
interval, when its $\levy$ measure is $\cf{0<x\le r^+} \gamma^+
e^{-q^+ x} x^{-1-\alpha^+}\,\dd x + \cf{-r^-\le x < 0} \gamma^-
e^{-q^- x} |x|^{-1-\alpha^-}\,\dd x + \chi(\dd x)$ with $\alpha^\pm\in
(0,1)$.  We omit detail on this. 

Let $S$ be a stable subordinator with $\levy$ measure $\Lambda$, so
that for $\lambda$, $t>0$, $\mean[e^{-\lambda S(t)}] =
\exp\{-t\gamma\Gamma(1-\alpha)\alpha^{-1} \lambda^\alpha\}$.  By
scaling of time, assume $\gamma = \alpha/\Gamma(1-\alpha)$ without
loss of generality.  Then $S(1)$ is a ``standard'' stable variable
with p.d.f.
\begin{align}
  \begin{array}{c}
    \displaystyle
    f(x) = \frac{\alpha}{(1-\alpha)\pi}\int_0^\pi h(x,\theta)\,
    \dd\theta,
    \quad
    \text{where} \\[1em]
    \displaystyle
    h(x, \theta) = \cf{x>0} h_0(\theta)\, x^{-1/(1-\alpha)}
    \exp\{-h_0(\theta)\, x^{-\alpha/(1-\alpha)}\}, \quad
    \text{with} \\[1em]
    h_0(\theta) = \sin[(1-\alpha)\theta]
    [\sin(\alpha \theta)]^{\alpha/(1-\alpha)}
    (\sin \theta)^{-1/(1-\alpha)}.
  \end{array}
  \label{eq:h}
\end{align}
The sampling of $S(1)$ is well-known \cite{chambers:76, devroye:86b,
  zolotarev:66}.  Define function
\begin{align*}
  \psi(x) =
  \cf{x\not=0}x^{-1}(1-e^{-x}) + \cf{x=0}.
\end{align*}

Let $0<K\le\infty$ and $\fpt = \inf\{t>0: Z(t)>c(t)\}\wedge K$.  An
algorithm to sample $(\fpt, Z(\fpt-), \jp Z(\fpt))$ is as follows.

\begin{display}
  \item[*\ ] Set $T=H=D=0$, $A=K$, $b(\cdot)\equiv c(\cdot)$,
    $M_\alpha = (1-\alpha)^{1-1/\alpha} \alpha^{-1-1/\alpha}
    e^{-1/\alpha}$.
  \item Sample $(D,J)$ as in step~1 in Table \ref{a:sub-fpt}.
  \item Sample $S(1)$.  Set $t_1$ such that
    $t_1^{1/\alpha} S(1)= b(t_1)\wedge r$, $t = t_1\wedge
    D$, and $z=b(t)\wedge r$.
  \item If $t=t_1<D$, then set
    \begin{align*}
      w_0 = -\left.\frac{\dd (b(u)\wedge r)}{\dd u}
      \right|_{u=t}, 
      \quad
      w_1 = \frac{\gamma z^{1-\alpha}}{\alpha(1-\alpha)}
    \end{align*}
    and do the following steps.  (Note: w.p.~1, $b(u)\wedge r$ is
    differentiable at $t$ with a non-positive derivative.)
    \begin{enumerate}[topsep=0em,itemsep=0em]
    \item Sample $\iota\in\{0,1\}$
      such that $\pr\{\iota=i\} = w_i/(w_0+w_1)$.  If $\iota=0$, then
      set $s=z$, $v=0$; else sample $\beta_1\sim\dbeta(1, 1 -
      \alpha)$, $\beta_2 \sim\dbeta(\alpha,1)$, and set $s = \beta_1
      z$, $v = (z-s)/\beta_2$.
    \item Sample $\vartheta\sim \dunif(0,\pi)$ and $U\sim
      \dunif(0,1)$.  If $U > h((t)^{-1/\alpha} s,
      \vartheta)/M_\alpha$, then go back to step~3(a).
    \end{enumerate}
  \item If $t=D<t_1$, then sample $S(t)$ conditional on $S(t)\le z$.
    Set $s = S(t)$ and $v=0$.
  \item Set
    \begin{align*}
      C_k = \frac{[s^{1-\alpha} t \gamma q\Gamma(1-\alpha)]^k}
      {k!\Gamma(1+k(1-\alpha))}, \quad k\ge 0, \qquad
      C = \sum_{k=0}^\infty C_k.
    \end{align*}
    Then do the following steps.
    \begin{enumerate}[topsep=0em,itemsep=0em]
    \item Sample $\kappa\in \{0,1,2\ldots\}$, such that
      $\pr\{\kappa = k\} = C_k/C$.
    \item If $\kappa = 0$, set $x=s$, $\varrho=1$; else sample 
      $\beta\sim \dbeta(1, k(1-\alpha))$ and $(\eno \omega \kappa)
      \sim \ddir(1-\alpha, \ldots, 1-\alpha)$, and
      set $x=s\beta$, $\varrho = \prod_i \psi(q (s-x) \omega_i)$.
    \item Sample $\vartheta \sim\dunif(0,\pi)$ and $U\sim\dunif(0,1)$.
      If $M_\alpha U>\varrho e^{-q x} h(t^{-1/\alpha} x,
      \vartheta)$, then go back to step 4(a); else output $x$ to the
      next step as a random sample of $X_1(t-)$ conditional on
      $X_1(t-) + X_2(t-)=s$.
    \end{enumerate}
  \item  The rest is the same as steps~6--8 in
    Table~\ref{a:sub-fpt}.
\end{display}

We next justify the algorithm.  All its steps correspond 1-to-1 to
those in Table~\ref{a:sub-fpt}, and only steps~2, 3 and 5 contain new
detail.  Denote $a=b\wedge r$, which is regular.   With $t_1$ being
the unique solution to $t^{1/\alpha} S(1) = a(t)$, from
\eqref{eq:dist-ft} and scaling property of $S$, $\pr\{\fpt^S_a \le t\}
= \pr\{t^{1/\alpha} S(1) \ge a(t)\} = \pr\{t_1\le t\}$, leading to
step~2.  Given $\fpt^S_a=t_1$ and $\fpt^*= t$, where $\fpt^* =
\fpt^S_a \wedge D$, by step~3 in Table~\ref{a:sub-fpt}, if $t=t_1<D$,
then we need to sample $(S(t-), \jp S(t))$ conditional on $\fpt^S_a =
t$.  From Theorem \ref{thm:conditional}
\begin{align*}
  \pr\{S(t-)\in \dd s, \jp S(t)\in \dd v\gv \fpt^S_a=t\}
  \propto
  g_t(s)\,[w_0 \delta(\dd s - z)\,\delta(\dd v) 
  + w_1 \rho(s,v) \,\dd s\,\dd v],
\end{align*}
with $z=a(t)$, $w_0 = |a'(t)|$, $w_1 = \gamma z^{1-\alpha}/[\alpha
(1-\alpha)]$, and $\rho$ is the following p.d.f.
\begin{align} \label{eq:rho}
  \rho(s,v) = 
  \cf{0\le z-s<v} \alpha(1-\alpha)z^{-1+\alpha} v^{-1-\alpha}.
\end{align}
Define random vector $(\iota, \zeta, V)$, such that $\pr\{\iota=0\} =
1-\pr\{\iota=1\} = w_0/(w_0 + w_1)$, $\pr\{\zeta=z, 
V=0\gv \iota=0\}=1$, $\pr\{\zeta\in\dd s, V\in\dd v\gv \iota=1\} =
\rho(s,v)\,\dd s\,\dd v$.  Let $\vartheta\sim \dunif(0,\pi)$ be
independent of $(\iota, \zeta, V)$.  Then by $\pdf_t(s) =
t^{-1/\alpha} f(t^{-1/\alpha} s)$, where $f$ given in \eqref{eq:h},
\begin{align}
  &
  \pr\{S(t-)\in \dd s, \jp S(t)\in \dd v\gv \fpt^S_a=t\} 
  \nonumber
  \\
  &\hspace{1cm}
  \propto \int h(t^{-1/\alpha} s, \theta) 
  \pr\{\vartheta\in \dd \theta,\iota\in \dd i, \zeta\in\dd s,
  V\in\dd v\}, \label{eq:mu-integral}
\end{align}
with the integral only over $\theta$ and $i$.  It is seen that
$\zeta\sim (1-U_1^{1/(1-\alpha)})z$ and conditional on $\zeta$, $V\sim
(z-\zeta) U_2^{-1/\alpha}$, with $U_1$, $U_2$ i.i.d.\ $\sim
\dunif(0,1)$.  Thus step 3(a) samples $(\vartheta, \iota, \zeta, V)$.
Next, for $x>0$ and $\theta\in (0,\pi)$, by change of variable
$s=x^{-\alpha/(1-\alpha)}$ in the expression of $h$,
\begin{align*}
  h(x, \theta) \le
  \sup_{\theta\in (0,\pi)} \Sbr{h_0(\theta)
    \times\sup_{s>0} (s^{1/\alpha} e^{-h_0(\theta) s})
  }
  = \alpha^{-1/\alpha} e^{-1/\alpha}
  \sup_{\theta\in (0,\pi)} h_0(\theta)^{1-1/\alpha}.
\end{align*}
Since $\sin(t\theta) / \sin(\theta) \ge t$ for $\theta\in (0,\pi)$ and
$t\in (0,1)$, $h_0(\theta)\ge (1-\alpha) \alpha ^{\alpha /
  (1-\alpha)}$, giving $h(x, \theta) \le M_\alpha$.  Thus,
step 3 is a rejection sampling procedure for the distribution
which is proportional to $h(t^{-1/\alpha} s, \theta) \pr\{\vartheta
\in\dd \theta, \iota\in \dd i, \zeta \in \dd s, V\in \dd v\}$.  Then
by \eqref{eq:mu-integral}, $(s,v)$ is a sample of $(S(t-), \jp S(t))$
conditional on $\fpt_a^S=t$.   The entire step~3 is now justified.

By step~5 in Table \ref{a:sub-fpt}, given $(\fpt^*, S(\fpt^*-)) = (t,
s)$ with $s\in (0,r]$, we need to sample $X_1(t)$ conditional on
$X_1(t)+X_2(t)=s$, where $X_1$ and $X_2$ are independent subordinators
with $\levy$ measures $\cf{x\le r} e^{-q x} \Lambda(\dd x)$ and
$\cf{x\le r} (1-e^{-q x}) \Lambda(\dd x)$, respectively.  By
\eqref{eq:tilting},
\begin{align*}
  \pr\{X_1(t)\in \dd x\gv X_1(t)+X_2(t)=s\}
  \propto e^{-q x} g_t(x)
  \pr\{X_2(t)\in s-\dd x\}, \quad 0\le s\le s.
\end{align*}
Since $X_2(t)$ is compound Poisson with $\levy$ density $\lambda(x) =
\cf{0<x\le r} t\gamma (1-e^{-q x}) x^{-1-\alpha}$,
\begin{align*}
  \pr\Cbr{X_2(t)\in s-\dd x}
  \propto \cf{0\le x\le s}\Sbr{
    \delta(s-\dd x) + \sum_{k=1}^\infty
    \frac{\lambda^{*k}(s-x)\,\dd x}{k!}
  },
\end{align*}
where $\lambda^{*k}$ is the $k$-fold convolution of $\lambda$.  For
$w>0$ and $k>1$, 
\begin{align*}
  \lambda^{*k}(w)
  &
  = w^{k-1}   \int \prod_{i=1}^k \lambda(w v_i)\,\sigma_k(\dd v),
\end{align*}
where $\sigma_k$ is the measure specified in Section
\ref{sec:prelim}.  Since $0\le w\le s \le r$, by the definition of
$\psi$ and Dirichlet distribution, for any $v=(\eno v k)$ with $v_i\ge
0$ and $\norm v=1$,
\begin{align*}
  \prod_{i=1}^k \lambda(w v_i) 
  = 
  (t\gamma)^k \prod_{i=1}^k \frac{1-e^{-q w v_i}}{(w v_i)^{1+\alpha}}
  &= 
  (t\gamma)^k q^k w^{-k\alpha} \prod_{i=1}^k \psi(q w v_i)
  \prod_{i=1}^k \nth{v_i^\alpha} \\
  &
  = w^{-k\alpha}
  \frac{[t\gamma q\Gamma(1-\alpha)]^k}{\Gamma(k(1-\alpha))}
  f_k(v) \prod_{i=1}^k \psi(q w v_i),
\end{align*}
where $f_k$ denotes the p.d.f.\ of $\ddir(\eno a k)$, with all
$a_i=1-\alpha$.  Denote by $\omega_k = (\omega_{k1}, \ldots, \omega_{k
  k})$ a random variable following the Dirichlet distribution.  Then
\begin{align*}
  \lambda^{*k}(s-x)\,\dd x
  = 
  (s-x)^{k(1-\alpha)-1}\,\dd x \times
  \frac{[t\gamma q\Gamma(1-\alpha)]^k}{\Gamma(k(1-\alpha))}
  \mean\Sbr{\prod_{i=1}^k \psi(q (s-x) \omega_{k i})}.
\end{align*}
Note that $\cf{0\le x\le s} k(1-\alpha)(s-x)^{k(1-\alpha)-1} /
s^{k(1-\alpha)}$ is the p.d.f.\ of $s\beta_k$, where $\beta_k\sim
\dbeta(1, k(1-\alpha))$.  Then, with $C_k$ the same as in the
algorithm, we get
\begin{align*}
  \lambda^{*k}(s-x)\,\dd x = k! C_k \pr\{s\beta_k\in \dd x\}
  \mean\Sbr{\prod_{i=1}^k \psi(q (s-x) \omega_{k i})}.
\end{align*}
The above identity also holds for $k=1$.  Combining with \eqref{eq:h},
we get
\begin{align*}
  &
  \pr\{X_1(t)\in \dd x\gv X_1(t)+X_2(t)=s\}\\
  &
  \propto e^{-q x} g_t(x)
  \Cbr{
    \delta(s-\dd x) + 
    \sum_{k=1}^\infty C_k \pr\{s\beta_k \in \dd x\}
    \mean\Sbr{\prod_{i=1}^k   \psi(q(s-x) \omega_{k i})}} \\
  &
  \propto \int_0^\pi e^{-q x} h(t^{-1/\alpha} x, \theta)\,\dd\theta
    \Cbr{
    \delta(s-\dd x) + 
    \sum_{k=1}^\infty C_k \pr\{s\beta_k \in \dd x\}
    \mean\Sbr{\prod_{i=1}^k \psi(q(s-x) \omega_{k i})}}.
\end{align*}
Starting from the last integral expansion, the treatment is similar to
step~3.  Define random vector $(\kappa, \zeta, \omega)$, such that
$\kappa\in \{0,1,2,\ldots\}$ with $\pr\{\kappa = k\} \propto C_k$,
conditional on $\kappa=0$, $\zeta=s$, $\omega=0$, and conditional on
$\kappa=k\ge 1$, $\zeta\sim s\beta_k$ and $\omega\sim \omega_k$ are
independent.  For any $x\in [0,s]$ and $w\in \cup_{k\ge 1}\Reals^k$,
define $\varrho(x, w) = \prod_{i=1}^k \psi(q(s-x) w_i)$, with $k$
equal to the dimension of $w$.  Finally, let $\vartheta\sim
\dunif(0,\pi)$ be independent from $(\kappa, \zeta, \omega)$.  Then
\begin{align*}
  &
  \pr\{X_1(t)\in \dd x\gv X_1(t)+X_2(t)=s\}\\
  &\quad\quad
  \propto \int e^{-q x} h(t^{-1/\alpha} x, \theta) \varrho(x,w)
  \pr\{\vartheta\in\dd \theta, \kappa\in\dd k, \zeta\in\dd x,
  \omega\in \dd w\}, 
\end{align*}
where the integral is only over $\theta$, $k$, and $w$.  Based on
this, it is seen step 5 is a rejection sampling procedure of $X_1(t)$
conditional on $X_1(t)+X_2(t)=s$.

\section{Finite mixture of exponentially tilted upper
  truncated stable L\'evy densities} 
\label{sec:mix-stable}
Given $I\ge 2$, $r, \eno\gamma I>0$, $q\ge 0$, and $\eno\alpha
I\in (0,1)$, let
\begin{align} \label{eq:mix-stable}
  \Ldf(\dd x)= \sum_{i=1}^I \ldf_i(x)\,\dd x + \chi(\dd x)
  \ \ \text{with}\ \ \ldf_i(x) = \cf{0<x\le
    r}\gamma_i e^{-q x} x^{-1-\alpha_i}.
\end{align}
The seemingly more general case where different $\ldf_i(x)$ have
different $r_i$ and $q_i$ is actually covered by
\eqref{eq:mix-stable}, once we set $\tilde r = \min r_i$, $\tilde q = 
\max q_i$, $\tilde\ldf_i(x) = \cf{1<x\le \tilde r} \gamma_i e^{-\tilde
  q x}  x^{-1-\alpha_i}$ and $\tilde\chi(\dd x) = \chi(\dd x) + \sum_i
[\ldf_i(x) - \tilde\ldf_i(x)]\,\dd x$.  As in last section, we shall
focus on the sampling for subordinators.  The extension to $\levy$
measures $\sum_{i=1}^I \ldf_i(\pm x)\,\dd x+\chi(\dd x)$ is rather
straightforward.

To apply the procedure in Table \ref{a:sub-fpt}, let $X_{i j}$,
$i=1,\ldots, I$, $j=1,2,3$, and $Q$ be independent subordinators with
no drift, such that $X_{i1}$, $X_{i2}$, and $X_{i3}$ have $\levy$
densities $\ldf_i(x)$, $\cf{0<x\le r} \gamma_i (1-e^{-q x})
x^{-1-\alpha_i}$, $\cf{x>r} \gamma_i x^{-1-\alpha_i}$, respectively,
and $Q$ has $\levy$ measure $\chi$.  Let $S_i = X_{i1} + X_{i2} +
X_{i3}$ and $Z = \sum_i X_{i1} + Q$.  Then $\eno S I$ are
independent stable processes with $\levy$ densities $\gamma_i x^{-1 -
  \alpha_i}$, respectively, and no drift, and $Z$ a subordinator with
$\levy$ measure $\Ldf$ and no drift.  Let $S = (\eno S I)$.  
Unlike previous sections, $S$ is multidimensional.  Let
$\Sigma=\norm S = \cum S I$.  Then $\Sigma(t) \sim \sum_{i=1}^I
t^{1/\alpha_i} S_i(1)$.  Let $d_i = [\alpha_i/\gamma_i
\Gamma(1-\alpha_i)]^{1/\alpha_i}$.  Then each $d_i S_i(1)$ has Laplace
transform $\mean[e^{-\lambda d_i S_i(1) }] =
\exp(-\lambda^{\alpha_i})$, $\lambda>0$, whose p.d.f.\ $f_i$ is given
in \eqref{eq:h}.  Let be $h_i$ the function in the integral
representation of $f_i$ in \eqref{eq:h}.

Let $0<K\le\infty$ and $\fpt = \inf\{t>0: Z(t)>c(t)\}\wedge K$.  An
algorithm to sample $(\fpt, Z(\fpt-), \jp Z(\fpt))$ is as follows.
\begin{display}
  \item[* ] Set $T=H=D=0$, $A=K$, $b(\cdot)\equiv c(\cdot)$, $M_i =
    (1-\alpha_i)^{1-1/\alpha_i}\alpha_i^{-1-1/\alpha_i}
    e^{-1/\alpha_i}$, $i\le I$
  \item Sample $(D,J)$ as in step~1 in Table \ref{a:sub-fpt}.
  \item Sample $S(1)$.  Set $t_1$ such that $\sum_{i=1}^I
    t_1^{1/\alpha_i} S_i(1) = b(t_1)\wedge r$, $t = t_1\wedge D$, and
    $z=b(t)\wedge r$.
  \item If $t=t_1<D$, then set
    \begin{align*}
      w_0 = -\nth{\Gamma(I)}
      \left.\frac{\dd (b(u)\wedge r)}{\dd u}\right|_{u=t},
      \quad
      w_i = \frac{\gamma_i z^{1-\alpha_i} \Gamma(1-\alpha_i)}{\alpha_i
        \Gamma(I+1-\alpha_i)}, \quad 1\le i\le I
    \end{align*}
    and do the following steps.
    \begin{enumerate}[topsep=0em, itemsep=0em]
    \item Sample 
      $\iota\in \{0,1,\ldots, I\}$, such that $\pr\{\iota=j\} =
      w_j/(w_0+w_1+\cdots w_I)$.  Sample $\omega=(\eno \omega I) \sim
      \ddir(1, \ldots, 1)$.  If $\iota=0$, then set $s_i= z\omega_i$,
      $i\le I$, and $v=0$.  If $\iota\ge 1$, then sample
      $\beta\sim \dbeta(I, 1-\alpha_\iota)$, $\beta'\sim
      \dbeta(\alpha_\iota, 1)$, and set $s_i = z\beta\omega_i$,
      $i\le I$, and $v = (z-s_1-\cdots - s_I)/\beta'$.
    \item Sample $\eno \vartheta I\sim \dunif(0,\pi)$ and
      $U\sim\dunif(0,1)$ independently.  If $U\ge \prod_{i=1}^I
      [h_i(d_i t^{-1/\alpha_i} s_i, \vartheta_i)/M_i]$, then go back
      to step 3(a).
    \end{enumerate}
  \item If $t=D<t_1$, then sample $S_1(1)$, $\ldots$, $S_I(1)$
    conditional on $\sum_{i=1}^I D^{1/\alpha_i} S_i(1)\le z$.
    Set $s=\sum_{i=1}^I D^{1/\alpha_i} S_i(1)$ and $v=0$.
  \item For $i\le I$, sample $x_{i1}$ from the distribution of
    $X_{i1}(t-)$ conditional on $X_{i1}(t-) +
    X_{i2}(t-) =s_i$, by executing step~5 in the
    algorithm in Section \ref{sec:stable} with the corresponding 
    parameters.  Set $x = x_{11}+x_{21} + \cdots + x_{I1}$.
  \item The rest is the same as steps 6--8 in
    Table~\ref{a:sub-fpt}.
\end{display}

To justify the algorithm, denote $a(t) = b(t)\wedge r$.  From
\eqref{eq:dist-ft} and scaling property of $S_i$, $\pr\{\fpt^\Sigma_a
\le t\}  = \pr\{\Sigma(t)\ge a(t)\} = \pr\{\sum_{i=1}^I t^{1/\alpha_i}
S_i(1) \ge a(t)\}$, which leads to step~2.

Given $\fpt^\Sigma_a=t_1$ and $\fpt^*=t$, where $\fpt^* =
\fpt^\Sigma\wedge D$, denote $z=a(t)$.  Denote by
$\pdf_t$ the p.d.f.\ of $\Sigma(t)$.  By Theorem
\ref{thm:conditional}, for $s=(\eno s I)$ with $s_i\ge 0$, and
$u, v\ge 0$,
\begin{align*}
  &
  \pr\{S(t-)\in \dd s,\, \Sigma(t-)\in\dd u,\,
  \Delta_\Sigma(t)\in \dd v\gv \fpt_a^\Sigma=t\}
  \\
  &\hspace{1cm}
  \propto\pdf_t(u)\,\pr\{S(t)\in \dd s\gv \Sigma(t) = u\}
  \\
  &
  \hspace{2cm}
  \times
  \Sbr{
    |a'(t)| \,\delta(\dd u - z)\delta(\dd v)
    + \cf{0\le z-u < v} 
    \sum_{i=1}^I \frac{\gamma_i\,\dd u\,\dd v}{v^{1+\alpha_i}}
  }.
\end{align*}
Since each $S_i(t)$ has p.d.f.\ proportional to $f_i(d_i
t^{-1/\alpha_i} x)$,
\begin{align*}
  g_t(u) \,\pr\{S(t)\in \dd s\gv \Sigma(t) = u\}
  \propto\prod_{i=1}^I f_i(d_i t^{-1/\alpha_i} s_i) \times
  \pr\{u\omega\in\dd s\}\times \frac{u^{I-1}}{\Gamma(I)},
\end{align*}
with $\omega=(\eno \omega I)\sim \ddir(1, \ldots, 1)$.  Let $w_0$,
$\eno w I$ be as in the algorithm.  Then
\begin{align*}
  \frac{u^{I-1}}{\Gamma(I)}
  \times |a'(t)|\, \delta(\dd u - z)\, \delta(\dd v)
  &= z^{I-1} w_0 \, \delta(\dd u-z)\, \delta(\dd v), \\
  \frac{u^{I-1}}{\Gamma(I)}\times \cf{0\le z-u < v}
  \frac{\gamma_i\,\dd u\,\dd v}{v^{1+\alpha_i}}
  &= z^{I-1}
  w_i\pr\{z\beta_i\in \dd u,\,
  (z-u)/\beta_i'\in \dd v\}, \quad i\ge 1,
\end{align*}
where $\beta_i\sim \dbeta(I, 1-\alpha_i)$ and $\beta_i'\sim
\dbeta(\alpha_i, 1)$ are independent.  The above identities together
with \eqref{eq:h} yield
\begin{align*}
  &
  \pr\{S(t-)\in \dd s,\, \Sigma(t-)\in\dd u, \jp\Sigma(t)\in\dd v
  \gv \fpt_a^\Sigma=t\} \\
  &\hspace{1cm}
  \propto
  \int_{\theta_i\in [0,\pi]}\prod_{i=1}^I h_i(d_i t^{-1/\alpha_i} s_i,
  \theta_i)\, \dd \theta_1\cdots\,\dd\theta_I\,
  \pr\{u\omega\in\dd s\}\\
  &\hspace{2cm}
  \times
  \Sbr{
    w_0\delta(\dd u-z)\,\delta(\dd v) + 
    \sum_{i=1}^I w_i \pr\{z\beta_i\in\dd u,\,
    (z-u)/\beta_i'\in\dd v\}
  }\!,
\end{align*}
where the integral is only over $\theta_i$'s.  Starting here the
treatment is similar to Section \ref{sec:stable}.  It is then seen
that step~3 samples $(s, v)\sim (S(t-), \jp \Sigma(t))$ conditional on
$\fpt^\Sigma_a=t<D$.

From step~4 in Table~\ref{a:sub-fpt} and the scaling property of
$S_i$, it is easy to see step~4 samples $\Sigma(D)$ conditional on
$D<\fpt^\Sigma_a$.  Finally, note the ultimate goal of step~5 in
Table~\ref{a:sub-fpt} is to sample $\sum_i X_{i1}(t-)$ conditional on
$\fpt^*=t$.  Since we have now sampled $(s,v)\sim (S(t-), \jp
\Sigma(t))$ conditional on $\fpt^*=t$, it suffices to sample
$(X_{11}(t-), \ldots, X_{I1}(t-))$ conditional on $(S(t-),
\jp\Sigma(t), \fpt^*)=(s, v, t)$.  Since in this case, $X_{i1}(t-) +
X_{i2}(t-) = s_i$ and $X_{i3}(t-)=0$, the conditional sampling is
equivalent to that of $X(t-)$, where $X=(X_{i j}, i=1,\ldots, I,
j=1,2,3)$.  By Theorem \ref{thm:joint}, if $t=t_1 < D$, then
\begin{align*}
  \pr\{X(t-)\in\dd x \gv \fpt^*=t,\, \Sigma(t-)=\norm s,\,
  \jp\Sigma(t)=v \}
  &= \pr\{X(t)\in \dd x \gv \Sigma(t) = \norm s\},\\
  \pr\{S(t-)\in\dd s \gv \fpt^*=t,\, \Sigma(t-)=\norm s,\,
  \jp\Sigma(t)=v\}
  &= \pr\{S(t)\in \dd s \gv \Sigma(t) = \norm s\},
\end{align*}
yielding $\pr\{X(t-)\in \dd x\gv \fpt^*=t,\,
S(t-)=s,\,\jp\Sigma(t)=v\} = \pr\{X(t)\in \dd x\gv S(t)=s\}$.  The
identity still holds if $t=D<t_1$.  By independence, the right hand
side is 
\begin{align*}
  \prod_{i=1}^I
  \pr\{X_{i1}(t)\in \dd x_{i1}, X_{i2}(t)\in \dd x_{i2}, X_{i3}(t)\in
  \dd x_{i3}\gv S_i(t)=s_i\},
\end{align*}
so to sample $(X_{11}(t-), \ldots, X_{I1}(t-))$ conditionally, it
suffices to sample $X_{i1}(t)$ independently, each conditional on
$S_i(t)=s_i$.  Therefore, step~5 in the algorithm in Section
\ref{sec:stable} can be used.  Once $X_{i1}(t)$ are sampled, their sum
is a sample value of $\sum_i X_{i1}(t)$.

\section{Upper truncated Gamma L\'evy density}
\label{sec:gamma}
Let $\Ldf(\dd x) = \ldf(x)\,\dd x + \chi(\dd x)$, where
\begin{align}
  \ldf(x) = \cf{0<x\le r} e^{-x} x^{-1}, \ r>0,
  \label{eq:levy-gamma}
\end{align}
and $\chi$ is a finite measure on $\prs$.  For a $\levy$ measure of
the more general form $\tilde\Ldf(\dd x) = 
\tilde\ldf(x)\,\dd x+ \tilde\chi(\dd x)$, where $\tilde\ldf(x) =
\cf{0<x\le \tilde r} \gamma e^{-q x} x^{-1}$, $\gamma$, $q>0$, the
sampling of the first passage event can be reduced to that for
$\Ldf$. Indeed, if $\tilde Z$ has $\levy$ measure $\tilde\Ldf$, then
$Z(t)=q \tilde Z(t/\gamma )$ has $\levy$ measure \eqref{eq:levy-gamma},
with $r=q \tilde r$, $\chi(\dd x) = \tilde\chi(\dd x/q)$ therein.
As a result, the first passage event of $\tilde Z$ across $\tilde
c$ can be deduced from that of $Z$ across $q\tilde
c(t/\gamma)$.

The sampling of the first passage event for $\Ldf$ is somewhat simpler
than those in previous sections, as exponential tilting is ``built
in'' the Gamma process.  Let $X_1$, $X_2$, and $Q$ be independent
subordinators with no drift and with $\levy$ measures $\ldf(x)\,\dd
x$, $\cf{x\ge r} e^{-x} x^{-1}\,\dd x$, and $Q$, respectively.  Then
$S=X_1+X_2$ is a Gamma process and $Z = X_1+Q$ has $\levy$ measure
$\Ldf$ and no drift.  Let $0<K\le\infty$ and $\fpt = \inf\{t>0:
Z(t)>c(t)\}\wedge K$.  An algorithm to sample $(\fpt, Z(\fpt-), \jp
Z(\fpt))$ is as follows.
\begin{display}
  \item[* ] Set $T=H=D=0$, $A=K$, $b(\cdot)\equiv c(\cdot)$
  \item Sample $(D,J)$ as in step~1 in Table \ref{a:sub-fpt}.
  \item Sample $U\sim \dunif(0,1)$.  Set $t_1$ such that
    $\int_{b(t_1)\wedge r}^\infty x^{t_1-1} e^{-x}\,\dd 
    x/\Gamma(t_1) = U$, $t = t_1\wedge D$, and $z=b(t)\wedge r$.
  \item If $t=t_1<D$, then set
    \begin{align*}
      w_0 &= -\left.\frac{\dd (b(u)\wedge r)}{\dd u}\right|_{u=t},
      \quad
      w_1 = \frac{2B(t,1/2) z}{e}, \quad
      w_2 = \nth t\\
      h_1(x,v)&
      = \cf{0\le z-x<v\le z}
      \frac{e^{z-x-v}(1-x/z)^{1/2}\ln[(1-x/z)^{-1}]}{2/e}, \\
      h_2(x,v)&
      = \cf{0\le z-x<z<v}\frac{z e^{-x}}{v}.
    \end{align*}
    and do the following steps.
    \begin{enumerate}[topsep=0em, itemsep=0em]
    \item Sample $\iota\in\{0,1\}$, such that $\pr\{\iota=i\} =
      w_i/(w_0 + w_1+w_2)$.  If $\iota=0$, then set $x=z$, $v=0$,
      $\eta=1$; if $\iota=1$, then sample $\beta\sim\dbeta(t,
      1/2)$, $\xi\sim \dunif(0,1)$, and set $x=z\beta$, $v =
      z(1-\beta)^\xi$, $\eta=h_1(x,v)$; if $\iota = 2$, then
      sample $\beta \sim \dbeta(t, 1)$, $\xi\sim\dexp(1)$,
      and set $x=z\beta$, $v=z+\xi$, $\eta=h_2(x,v)$.
    \item Sample $U\sim\dunif(0,1)$.  If $U>\eta$, then go back to
      step~3(a).
    \end{enumerate}
  \item If $t=D<t_1$, then sample $\gamma\sim\dgamma(D,1)$ conditional
    on $\gamma\le z$.  Set $x=\gamma$, $v=0$.
  \item The rest is the same as steps~6--8 in
    Table~\ref{a:sub-fpt}.
\end{display}

To justify the algorithm, denote $a(t) = b(t)\wedge r$.  From
\eqref{eq:dist-ft} and $S(t)\sim\dgamma(t,1)$,
\begin{align*}
  \pr\{\fpt^S_a \le t\}
  = \pr\{S(t)\ge a(t)\} = \nth{\Gamma(t)} \int_{a(t)}^\infty x^{t-1}
  e^{-x}\,\dd x,
\end{align*}
which is a continuous and strictly increasing function.  The sampling
in step~2 is then the standard inversion \cite[cf.][]{devroye:86b}.
Given $\fpt^S_a=t_1$ and $\fpt^*=t$, where $\fpt^* = \fpt^S_a\wedge
D$, by step~3 in Table~\ref{a:sub-fpt}, if $t=t_1<D$, then we need to
sample $(S(t-), \jp S(t))$ conditional on $\fpt^S_a=t$.  Let $\pdf_t$
denote the p.d.f.\ of $\dgamma(t,1)$.  Let $z=a(t)$.  From Theorem
\ref{thm:conditional}, for $x,v>0$,
\begin{align*}
  &
  \pr\{S(t-)\in \dd x, \jp S(t)\in \dd v\gv \fpt^S_a = t\} \\
  &\propto
  |a'(t)| g_t(z) \delta(\dd x - z) \delta(\dd v) 
  + \cf{0\le z-x<v} g_t(x) \frac{e^{-v}}{v}\,\dd x \,\dd v \\
  &
  \propto
  |a'(t)| \delta(\dd x - z)\delta(\dd v) + q_1(x,v)\,\dd x\,\dd v +
  q_2(x,v)\,\dd x\,\dd v,
\end{align*}
where, letting $q(x,v) = g_t(x)e^{-v}/[v g_t(z)]$, $q_1(x,v) =
\cf{0\le z-x<v\le z} q(x,v)$ and $q_2(x,v) = \cf{0\le z-x< z\le v}
q(x,v)$.  Now $q(x,v) = (x/z)^{t-1} e^{z-x-v}/v$.  Let
\begin{align*}
  \rho_1(x,v)
  &
  = \cf{0\le z-x<v\le z}
  \frac{(x/z)^{t-1} (1-x/z)^{-1/2}}{B(t,1/2) z} \,
  \nth{v \ln[(1-x/z)^{-1}]}, \\
  \rho_2(x,v)
  &
  = \cf{0\le z-x<z\le v} \frac{t(x/z)^{t-1} e^{z-v}}{z}.
\end{align*}
For $i=1,2$, $\rho_i(x,v)$ is a p.d.f.\ and $q_i(x,v) = w_i h_i(x,v)
\rho_i(x,v)$, where $w_i$ and $h_i$ are defined in the algorithm.  It
is easy to check $h_i(x,v)\le 1$ for $i=1,2$.  Define  $h_0(x,v)\equiv
1$.  Define $(\iota, \zeta, V)$ such that $\iota\in \{0,1,2\}$ with
$\pr\{\iota = i\} = w_i/(w_0+w_1+w_2)$, conditional on $\iota=0$,
$\zeta = z$ and $V=0$, and conditional on $\iota=i\in \{1,2\}$,
$(\zeta, V)$ has p.d.f.\ $\rho_i$.  Then
\begin{align*}
  \pr\{S(t-)\in\dd x,\, \jp S(t)\in \dd v\gv \fpt^S_a =t\}
  = \int h_i(x,v) \pr\{\iota\in\,\dd i, \zeta\in\,\dd x, V\in\,\dd
  v\},
\end{align*}
where the integral is only over $i$.  It is easy to check that when
$(\zeta, V)$ has p.d.f.\ $\rho_1$, then $(\zeta, V) \sim (z\beta_t,
z(1-\beta_t)^U)$, where $\beta_t\sim \dbeta(t,1/2)$ and $U\sim
\dunif(0,1)$, and when $(\zeta, V)$ has p.d.f.\ $\rho_2$, then
$(\zeta, V)\sim (z\beta_t', z+\xi)$, where $\beta_t'\sim \dbeta(t,1)$
and $\xi\sim \dexp(1)$.  Then step~3 in the algorithm is a rejection
sampling procedure of $(S(t-), \jp S(t)$ conditional on $\fpt^S_a=t$.
Since $S(t-)$ sampled by step~3 or 4 is exactly $X_1(t-)$, there is no
need for a counterpart of step~5 in Table~\ref{a:sub-fpt}.  We can
directly go to steps~6--8 in Table~\ref{a:sub-fpt}.

\setcounter{equation}{0}
\renewcommand{\theequation}{A.\arabic{equation}}
\renewcommand{\thetheorem}{A.\arabic{theorem}}

\section*{Appendix}
To prove Lemmas \ref{lemma:cdc1} and \ref{lemma:cdc2}, we start with
two more lemmas. 
\begin{lemma} \label{lemma:L-u-c}
  For $t>0$ and $0<a<b<\infty$, let
  \begin{align*}
    L_1(t,a,b)
    =\int_0^a [\dual\Ldf_S(a-u) - \dual\Ldf_S(b-u)] \pdf_t(u)\,\dd u,
    \quad
    L_2(t, a, b)
    =\int_a^b \dual\Ldf_S(b-u) \pdf_t(u)\,\dd u.
  \end{align*}
  Then for any $E=[t_0, t_1]\subset\prs$ and $I=[\alpha,
  \beta]\subset\prs$,
  \begin{align} \label{eq:L-u-c}
    \lim_{r\dt 0} \sup\Cbr{
    L_i(t,a,b): t\in E,\, a,b\in I,\, 0\le b-a\le r} =0, \quad i=1,2.
  \end{align}
\end{lemma}
\begin{proof}
  Since
  $\dual\Ldf_S$ is non-increasing, for $0\le b-a\le r$, $L_1(t,a,b) \le
  L_1(t,a,a+r)$.  Fix $\rx\in (0,\alpha/2)$ and let $h(u) =
  [\dual\Ldf_S(a-u) -  \dual\Ldf_S(a+r-u)]\pdf_t(u)$.  Then 
  $L_1(t,a,a+r) = J_1 + J_2$, where $J_1=J_1(t,a,r) =  \int_0^\rx h$,
  $J_2=J_2(t, a,r) = \int_\rx^a h$.  We have
  \begin{align*}
    J_1 \le \int_0^\rx \dual\Ldf_S(a-u)\pdf_t(u)\,\dd u
    &
    \le \dual\Ldf_S(a-\rx) \int_0^\rx \pdf_t(u)\,\dd u\\
    &= \dual \Ldf_S(\alpha - \rx) \pr\Cbr{S(t)\le \rx}
    \le \dual \Ldf_S(\alpha - \rx) \pr\Cbr{S(t_0)\le \rx}
  \end{align*}
  and letting $M = \sup\{\pdf_t(u): t\in E,\, \rx\le u\le\beta\}$,
  \begin{align*}
    J_2 \le
    M\int_\rx^a [\dual\Ldf_S(a-u) - \dual\Ldf_S(a-u+r)] \,\dd u
    \le M\int_0^\beta [\dual\Ldf_S(u)-\dual\Ldf_S(u+r)]\,\dd u,
  \end{align*}
  By the assumption on $\pdf_t(x)$, $M<\infty$.  Also, $\int_0^\beta
  \dual\Ldf_S(u)\,\dd u  = \int_0^\infty (v\wedge\beta) \Ldf_S(\dd v)
  < \infty$.  Then by monotone convergence, as $r\dt 0$, $J_2\to 0$
  uniformly for $(t,a)\in E\times I$.  As a result, for $L_1$, the
  limit in \eqref{eq:L-u-c} is no greater than $\dual\Ldf_S(\alpha -
  \rx) \pr\Cbr{S(t_0)\le \rx}$.  Since $\pr\{S(t_0)>0\}=1$ and $\rx$
  is arbitrary, the limit is equal to 0.  Thus \eqref{eq:L-u-c} holds
  for $L_1$.

  Next, fixing $\rx\in (0,\alpha)$, by change of variable,
  for $t\in E$, $a$, $b\in I$ with $0\le b-a\le r$,
  \begin{align*}
    L_2(t,a,b)
    = \int_0^\infty \Ldf_S(\dd x) \int_{a\vee(b-x)}^b
    \pdf_t 
    &\le \int_0^\rx \Ldf_S(\dd x)\int_{b-x}^b \pdf_t
    + \int_\rx^\infty \Ldf_S(\dd x) \int_a^b \pdf_t \\
    &\le
    M' \Sbr{\int_0^\rx x\Ldf_S(\dd x) +r \dual\Ldf_S(\rx)},
  \end{align*}
  where $M' = \sup\{\pdf_t(u): t\in E, \alpha - \rx\le u\le\beta\}$.
  Therefore, as $r\dt 0$, the limit for $L_2$ in \eqref{eq:L-u-c} is
  no greater than $M' \int_0^\rx x\Ldf_S(\dd x)$.  Since $\rx$ is
  arbitrary, the limit is 0. 
\end{proof}

\begin{lemma} \label{lemma:H}
  Let $h$ be a bounded function on $\prs\times \prs$.  For $a$,
  $t\in\prs$, define
  \begin{align*}
    H(a,t) = \int \cf{u\le a< x}
    h(u,x)\,\pr\!\Cbr{S(t)\in\,\dd u}\Ldf_S(\dd x-u).
  \end{align*}
  Then under the continuous density condition, $H$ is continuous on
  $\prs\times\prs$.
\end{lemma}
\begin{proof}
  Without loss of generality, suppose $|h(u,x)|\le 1$.  It suffices to
  show $H$ is continuous on any $R=[\alpha, \beta]\times [t_0,
  t_1]\subset \prs\times \prs$.  Let $(a,s)$, $(b,t)\in R$.  Then
  \begin{align*}
    |H(b,t)-H(a,s)|\le |H(b,t)-H(a,t)|+|H(a,t)-H(a,s)|
  \end{align*}
  Let $L_1$ and $L_2$ be as in Lemma \ref{lemma:L-u-c}.  Let
  $a'=a\wedge b$ and $b'=a\vee b$.  Then
  \begin{multline*}
    |H(b,t)-H(a,t)| 
    \le\int
    |\cf{u\le b < x} - \cf{u\le a< x}|\,
    \pr\Cbr{S(t)\in\dd u}\Ldf_S(\dd x-u)\\
    \le \int \Grp{
      \cf{u\le a'< x\le b'} + \cf{a'< u\le b'< x}
    }
    \pr\Cbr{S(t)\in\dd u} \Ldf_S(\dd x-u).
  \end{multline*}
  The right hand side is $L_1(t,a',b') + L_2(t,a',b')$.  Then by Lemma
  \ref{lemma:L-u-c}, as $(b,t)\to (a,s)$,  $L_1(t,a',b') +
  L_2(t,a',b')\to 0$, giving $H(b,t) - H(a,t)\to 0$.  

  It only remains to show $H(a,t)-H(a,s)\to 0$.  Given $\rx\in
  (0,\alpha)$, let $M=\sup\{\pdf_t(u): u\in [\alpha-\rx, \beta], t\in
  [t_0, t_1]\}$.  Then
  \begin{align*}
    |H(a,t)-H(a,s)|
    &\le \int\cf{u\le a<x} |\pdf_t(u)-\pdf_s(u)|\Ldf_S(\dd x-u)\,\dd u\\
    &
    = \int\cf{u\le a} |\pdf_t(u)-\pdf_s(u)|\dual\Ldf_S(a-u)\,\dd u.
  \end{align*}
  Bounding the integral on $[a-\rx,a]$ and $[0,a-\rx]$ separately,
  we obtain
  \begin{align*}
    |H(a,t)-H(a,s)|\le 2 M \int_0^\rx \dual\Ldf_S(u)\,\dd u
    + \dual\Ldf_S(\rx) \int |\pdf_t(u) - \pdf_s(u)|\,\dd u,
  \end{align*}
  Let $t\to s$.  Since point-wise convergence of $\pdf_t$ to $\pdf_s$
  implies convergence in total variation, $\Lsup|H(a,t)-H(a,s)|\le 2
  M \int_0^\rx\dual\Ldf_S(u)\,\dd u < \infty$.  Letting $\rx\to 0$
  gets $H(a,t)-H(a,s)\to 0$.
\end{proof}

\begin{proof}[Proof of Lemma \ref{lemma:cdc1}]
  Apply \eqref{eq:time-density} and Lemma \ref{lemma:H}, with
  $h(a,t)\equiv 1$ therein.
\end{proof}

\begin{proof}[Proof of Lemma \ref{lemma:cdc2}]
  Let $G_c = \Cbr{t>0: c(t)>0}$.  Then $G_c$ is a non-empty open
  interval and $\pr\Cbr{\fpt_c\in G_c}=1$.  To prove the lemma,
  it suffices to show that for any $0<t_0 < t_1<\infty$ with
  $[t_0, t_1]\subset G_c$ and $A\subset (t_0, t_1)$ with $\ell(A) =
  c(A) =0$, $\pr\Cbr{\fpt_c\in A}=0$.  Let $\alpha = c(t_1)$ and $\beta =
  c(t_0)$.  Given $\rx>0$, $A$ can be covered by at most countably
  many disjoint intervals $(a_i, b_i)\subset (t_0, t_1)$ such that
  $\sum (b_i-a_i)<\rx$ and $\sum [c(a_i) - c(b_i)]<\rx$.   For each
  $i$,
  \begin{align*}
    \pr\Cbr{\fpt_c\in (a_i,b_i)}
    &
    \le \pr\Cbr{S(a_i)\le c(a_i),\, S(b_i)>c(b_i)}\\
    &
    \le \pr\Cbr{c(b_i)<S(a_i)\le c(a_i)} + \pr\Cbr{S(a_i)\le
      c(b_i)<S(b_i)}
    \\
    &=
    \pr\Cbr{c(b_i)<S(a_i)\le c(a_i)} + \pr\Cbr{\fpt_{c(b_i)} \in
      (a_i,b_i)}.
  \end{align*}
  Since $a_i, b_i\in [t_0, t_1]$ and $c(a_i), c(b_i)\in [\alpha,
  \beta]$, by the continuous density condition and Lemma
  \ref{lemma:cdc1}, the right hand side is no greater than $M_1
  [c(a_i)-c(b_i)] + M_2 (b_i-a_i)$, where $M_1 = \sup \pdf_t(x)$ and
  $M_2=\sup \psi_x(t)$ over $(t,x)\in [t_0, t_1]\times [\alpha,
  \beta]$.  Therefore,  $\pr\Cbr{\fpt_c\in A} \le \sum
  \pr\Cbr{\fpt_c\in (a_i,b_i)} \le (M_1+M_2)\rx$. Since $\rx$ is
  arbitrary, this yields the proof.
\end{proof}

\bibliographystyle{acmtrans-ims}
\bibliography{Levy,ldpdb,MCMC}

\end{document}